\documentclass[11pt]{amsart}
\usepackage{graphicx}
\usepackage{amsmath} 
\usepackage{amsthm,amsfonts,amssymb,mathrsfs,amscd,amstext,amsbsy} 
\usepackage{epic,eepic} 
\usepackage{yfonts}
\usepackage{paralist,enumerate}
\usepackage[all]{xy}
\usepackage{hyperref}
\hypersetup{colorlinks}

\newcommand{\D}{\displaystyle}

\newtheorem{theorem}{Theorem}[section]
\newtheorem{cor}[theorem]{Corollary}

\newtheorem{theo}[theorem]{Theorem}
\newtheorem{lem}[theorem]{Lemma}
\newtheorem{pro}[theorem]{Proposition}
\newtheorem{rem}[theorem]{Remark}

\newtheorem{Definition}[theorem]{Definition}
\newtheorem*{Definition*}{Definition}

\def\qed{\hfill \ifhmode\unskip\nobreak\fi\quad\ifmmode\Box\else$\Box$\fi\\ }

\begin{document}

\title[Symplectic circle actions with three fixed points]{Symplectic periodic flows with exactly three equilibrium points}
\author{Donghoon Jang}
\address{Department of Mathematics, University of Illinois at Urbana-Champaign, Urbana, IL 61801}
\email{jang12@illinois.edu}

\begin{abstract}
Let the circle act symplectically on a compact, connected symplectic manifold $M$. If there are exactly three fixed points, $M$ is equivariantly symplectomorphic to $\mathbb{CP}^{2}$.
\end{abstract}

\maketitle

\section{Introduction}

$\indent$

The study of fixed points of flows or maps is a classical and important topic studied in geometry and dynamical systems. Torus actions in symplectic geometry correspond to periodic flows in mechanical systems. Fixed points by the actions correspond to equilibrium points by the flows. If a torus action has fixed points, a lot of information is encoded by the fixed point set of the action.

Hamiltonian actions are symplectic but symplectic actions need not be Hamiltonian. Hence it is a natural question to ask if there is a non-Hamiltonian symplectic action on a compact symplectic manifold when given a certain number of fixed points. It is a classical fact that a Hamiltonian circle action on a compact symplectic manifold $(M,\omega)$ has at least $\frac{1}{2}\dim M+1$ fixed points. 

Let the circle act symplectically on a compact, connected symplectic manifold $M$. First of all, there cannot be exactly one fixed point. Also, due to C. Kosniowski, if there are exactly two fixed points, then either $M$ is the 2-sphere or $\dim M=6$ \cite{Ko}. This is reproved by A. Pelayo and S. Tolman using another method \cite{PT}. This result itself does not rule out the possibility that there is a $6$-dimensional compact symplectic manifold $M$ with exactly two fixed points by the symplectic circle action. In this paper, we follow the method that A. Pelayo and S. Tolman use.

In this paper we study symplectic circle actions on compact, connected symplectic manifolds with exactly three fixed points. The conclusion is that in this case the manifold must be 4-dimensional and it is equivariantly symplectomorphic to $\mathbb{CP}^{2}$ with the standard action. In particular, if there are exactly three fixed points, symplectic circle actions are Hamiltonian.

In the proof, we use the equivariant cohomology of $M$ and weights at each fixed point. ABBV localization Theorem is used, but only the image of 1 is used. If there is a weight $a$ for some fixed point $p$, then there also is a weight $-a$ for some fixed point. If two distinct fixed points lie in the same component $N$ of $M^{\mathbb{Z}_{k}}$ for some $k>1$, then weights at the two fixed points are equal modulo $k$. The number of fixed points with $i$ negative weights is the same as the number of fixed points with $n-i$ negative weights, where $\dim M=2n$.

A two form $\sigma \in \Omega^{2}(M)$ on a compact manifold $M$ is called a \emph{symplectic form} if it is closed and non-degenerate. If a circle action on $M$ preserves $\sigma$, then the action is called \emph{symplectic}. Let $X_{M}$ be the vector field on $M$ induced by the circle action. Then the action is called \emph{Hamiltonian}, if there exists a map $\mu : M \rightarrow \mathbb{R}$ such that
\begin{center}
$-d\mu=\iota_{X_{M}}\sigma$.
\end{center}
This implies that every symplectic action is Hamiltonian if $H_{1}(M;\mathbb{R})=0$, since $\iota_{X_{M}}\sigma$ is closed. The main result of this paper is the following:

\begin{theo} \label{t11} Let the circle act symplectically on a compact, connected symplectic manifold $M$. If there are exactly three fixed points, $M$ is equivariantly symplectomorphic to $\mathbb{CP}^{2}$.
\end{theo}

\begin{proof}
By quotienting out by the subgroup which acts trivially, without loss of generality we may assume that the action is effective. Then this is an immediate consequence of Proposition \ref{p43}, Proposition \ref{p51}, and Proposition \ref{p61} below.
\end{proof}

\section{Backgrounds and Notations}

Consider a circle action on a manifold $M$. \emph{The equivariant cohomology of} $M$ is defined by $\D{H_{S^{1}}^{*}(M)=H^{*}(M \times_{S^{1}} S^{\infty})}$. If $M$ is oriented and compact, then from the projection map $\pi : M \times_{S^{1}} S^{\infty} \rightarrow \mathbb{CP}^{\infty}$ we obtain a natural push-forward map

\begin{center}
$\pi_{*} : \D{H_{S^{1}}^{i}(M;\mathbb{Z}) \rightarrow \D{H_{S^{1}}^{i-\dim M}(\mathbb{CP}^{\infty};\mathbb{Z})}}$
\end{center}
This map is given by "integration over the fiber" and denoted by $\int_{M}$.

\begin{theo} \label{t21} (ABBV localization) \cite{AB} Let the circle act on a compact oriented manifold $M$. Fix $ \alpha \in H_{S^{1}}^{ \ast } ( M ;  \mathbb{Q} ) $. As elements of $\mathbb{Q}(t)$,
\begin{center}
$\D{\int_{M} \alpha = \sum_{F \subset M^{S^{1}}} \int_{F} \frac{\alpha|_{F}}{e_{S^{1}(N_{F})}}}$,
\end{center}
where the sum is over all fixed components, and $e_{S^{1}} ( N_{F} )$  denotes the equivariant Euler class of the normal bundle to F.
\end{theo}

Let the circle act symplectically on a symplectic manifold $(M, \sigma)$. To each isolated fixed point $p \in M^{S^{1}}$, we can assign well-defined non-zero (integer) weights $\xi_{p}^{1}, ..., \xi_{p}^{n}$ in the isotropy representation $T_{p}M$(repeated with multiplicity). Denote $\sigma_{i}$ by the $i^{th}$ elementary symmetric polynomial. Then the $i^{th}$-equivariant Chern class is given by
\begin{center}
$c_{i}(M)|_{p}=\sigma _{i}(\xi_{p}^{1}, ..., \xi_{p}^{n})t^{i}$,
\end{center}
where $t$ is the generator of $H_{S^{1}}^{2}(p;\mathbb{Z})$. For instance, the first Chern class map at $p$ is given by $c_{1}(M)|p=\Sigma \xi_{p}^{i} t$, and the equivariant Euler class of the tangent bundle at p is given by $e_{S^{1}}(N_{F})=c_{n}(M)|_{p}=(\prod \xi_{p}^{i})t^{n}$. Hence,

\begin{center}
$\D{\int_{p} \frac{c_{i}(M)|_{p}}{e_{S^{1}} ( N_{p} )}=\frac{\sigma _{i}(\xi_{p}^{1}, ..., \xi_{p}^{n})}{\prod \xi_{p}^{j}}t^{i-n}}.$
\end{center}

Weights in the isotropy representation $T_{p}M$ also satisfy the following:

\begin{lem} \label{l22} \cite{PT}
Let the circle act symplectically on a compact symplectic $2n$-manifold with isolated fixed points. Then
\begin{center}
$\D{ | \{p \in M^{S^{1}} \mid \lambda_{p} = 2i \} | = | \{p \in M^{S^{1}} \mid \lambda_{p} = 2n-2i\} |}$, for all $i$,
\end{center}
where $\lambda_{p}$ is twice of the number of negative weights at $p$ for all $\D{p \in M^{S^{1}}}$.
\end{lem}

\begin{cor} \label{c23} \cite{PT}
Let the circle act symplectically on a compact symplectic $2n$-manifold with $k$ isolated fixed points. If $k$ is odd, then $n$ is even.
\end{cor}

\begin{lem} \label{l24} \cite{PT}
Let the circle act symplectically on a compact symplectic manifold $M$ with isolated fixed points. Then
\begin{center}
$\D{\sum_{p \in M^{S^{1}}} N_{p}(l)=\sum_{p \in M^{S^{1}}} N_{p}(-l)}$, for all $l \in \mathbb{Z}$.
\end{center}
Here, $N_{p} ( l ) $ is the multiplicity of $l$ in the isotropy representation $T_{p} M$ for all weights $l \in \mathbb{Z} $ and all $p  \in M^{S^{1}}$.
\end{lem}

\begin{lem} \label{l25} \cite{T}
Let the circle act on a compact symplectic manifold $(M,\omega)$. Let $p$ and $p'$ be fixed points which lie in the same component $N$ of $M^{\mathbb{Z}_{k}}$, for some $k>1$. Then the $S^{1}$-weights at $p$ and at $p'$ are equal modulo $k$.
\end{lem}

Let the circle act on a compact symplectic manifold $M$. Let $p$ and $p'$ be fixed points which lie in the same component $N$ of $M^{\mathbb{Z}_{k}}$, for some $k>1$. Denote $\Sigma_p$ and $\Sigma_{p'}$ by the multisets of weights at $p$ and $p'$, respectively. Lemma \ref{l25} states that there exists a bijection between $\Sigma_p$ and $\Sigma_{p'}$ that takes each weight $\alpha$ at $p$ to a weight $\beta$ at $p'$ such that $\alpha \equiv \beta \mod k$.

\begin{theo}\label{t26} \cite{PT}
Let the circle act symplectically on a compact, connected symplectic manifold $M$ with exactly two fixed points. Then either $M$ is the 2-sphere or $\dim M=6$ and there exist natural numbers $a$ and $b$ so that the weights at the two fixed points are $\{a, b, -a-b\}$ and $\{a+b, -a, -b\}$.
\end{theo}

\begin{cor} \label{c27} \cite{PT}
Let the circle act symplectically on a compact symplectic manifold $M$ with non-empty fixed point set. Then there are at least two fixed points, and if $\dim M \geq 8$, then there are at least three fixed points. Moreover, if the Chern class map is not identically zero and $\dim M \geq 6$, then there are at least four fixed points.
\end{cor}

\section{Preliminaries}

One of the main ideas to prove Theorem \ref{t11} is to look at the biggest weight among all the weights of fixed points.

Consider a symplectic circle action on a $2n$-dimensional compact, connected symplectic manifold $M$ with exactly three fixed points. The key fact is that the largest weight occurs only once. From this it follows that if without loss of generality we assume that $\lambda_p \leq \lambda_q \leq \lambda_r$ where $p,q$, and $r$ are the three fixed points, then $\lambda_p=n-2$, $\lambda_q=n$, and $\lambda_r=n+2$.

\begin{Definition} \label{d31} A weight $d$ is \textbf{the largest weight} if it is the biggest weight such that $\sum_{u \in M^{S^{1}}} N_{u}(d)>0$.
\end{Definition}

\begin{lem} \label{l32}
Let the circle act symplectically on a compact, connected symplectic manifold $M$ and suppose that there are exactly three fixed points. Let $d$ be the largest weight. Then the isotropy submanifold $M^{\mathbb{Z}_d}$ contains exactly two components that have fixed points: one isolated fixed point and one two-sphere that contains two fixed points.
\end{lem}

\begin{proof}
Consider the isotropy submanifold $M^{\mathbb{Z}_{d}}$. By Theorem \ref{t26}, the only possible cases are:
\begin{enumerate}
\item The isotropy submanifold $M^{\mathbb{Z}_d}$ contains a 2-sphere with two fixed points. The third fixed point is another component of $M^{\mathbb{Z}_d}$.
\item The isotropy submanifold $M^{\mathbb{Z}_e}$ contains a 6-dimensional component with two fixed points. The third fixed point is another component of $M^{\mathbb{Z}_d}$.
\item The isotropy submanifold $M^{\mathbb{Z}_d}$ contains a component with the three fixed points.
\end{enumerate}
The subset inclusions may not be equalities since $M^{\mathbb{Z}_{d}}$ may contain other components with no fixed points. 

Suppose that the second case holds. By Theorem \ref{t26}, the weights in the isotropy submanifold $M^{\mathbb{Z}_d}$ at two fixed points that lie in the 6-dimensional component are $\{a, b, -a-b\}$ and $\{-a, -b, a+b\}$ for some natural numbers $a$ and $b$. Moreover, $a, b$, and $a+b$ are multiples of $d$, which is impossible since $d$ is the largest weight.

Suppose that the third case holds. Let $Z$ be the component. Let $\dim Z=2m$. Since all the weights in the isotropy submanifold $M^{\mathbb{Z}_d}$ are either $d$ or $-d$, by Theorem \ref{t21},
\begin{center}
$\D{0=\int_{Z} 1=\frac{1}{\prod_{i=1}^{m} \pm d}+\frac{1}{\prod_{i=1}^{m} \pm d}+\frac{1}{\prod_{i=1}^{m} \pm d}=\pm\frac{1}{d^{m}}\mp\frac{1}{d^{m}}\pm\frac{1}{d^{m}}\neq 0},$
\end{center}
which is a contradiction.

Hence the first case is the case and the weights in the isotropy submanifold $M^{\mathbb{Z}_d}$ at the two fixed points in the 2-sphere are $\{-d\}$ and $\{d\}$.
\end{proof}

\begin{lem} \label{l33} Let the circle act symplectically on a $2n$-dimensional compact, connected symplectic manifold $M$ and suppose that there are exactly three fixed points. Then we can label the fixed points $p,q$, and $r$ so that $\lambda_{p}=n-2, \lambda_{q}=n$, and $\lambda_{r}=n+2$. Moreover, if $\dim M \neq 4$, then after possibly reversing the circle action, we may assume that $-d \in \Sigma_p$ and $d \in \Sigma_q$, where $d$ is the largest weight.
\end{lem}

\begin{proof} Let $p,q$, and $r$ be the fixed points. Without loss of generality, assume that $\lambda_p \leq \lambda_q \leq \lambda_r$. By Corollary \ref{c23}, $n$ is even. Also, since the number of fixed points is odd, $\lambda_{q}=n$ by Lemma \ref{l22}. Moreover, since $M$ is connected, $\dim M \neq 0$.

First, assume that $\dim M=4$. Then by Lemma \ref{l22}, either $\lambda_{p}=\lambda_{q}=\lambda_{r}=2$, or $\lambda_{p}=0, \lambda_{q}=2$, and $\lambda_{r}=4$. Suppose that $\lambda_{p}=\lambda_{q}=\lambda_{r}=2$. Then by Theorem \ref{t21},
\begin{center}
$\D{0=\int_{M} 1=\frac{1}{\prod_{i=1}^{2}\xi_{p}^{i}}+\frac{1}{\prod_{i=1}^{2}\xi_{q}^{i}}+\frac{1}{\prod_{i=1}^{2}\xi_{r}^{i}}<0},$
\end{center}
which is a contradiction. Hence $\lambda_{p}=0, \lambda_{q}=2$, and $\lambda_{r}=4$.

Next, assume that $\dim M \geq 8$. By Lemma \ref{l32}, we can label the fixed points $\alpha$, $\beta$, and $\gamma$ such that $\alpha$ and $\beta$ lie in the same 2-dimensional connected component of $M^{\mathbb{Z}_d}$ such that $-d \in \Sigma_{\alpha}$, $d \in \Sigma_{\beta}$, and $N_{\gamma}(d)=N_{\gamma}(-d)=0$. By reversing the circle action if necessary, we may assume that either $\lambda_{\alpha} \leq \lambda_{\gamma}$ or $\lambda_{\beta} \leq \lambda_{\gamma}$. Moreover, by Corollary \ref{c27}, the first Chern class map is identically zero. By Lemma \ref{l34} below, $\lambda_{\alpha}+2=\lambda_{\beta}$. Together with Lemma \ref{l22}, the above statements imply that $\lambda_p=n-2$, $\lambda_q=n$, $\lambda_r=n+2$, $-d \in \Sigma_p$, and $d \in \Sigma_q$.
\end{proof}

To prove Lemma \ref{l33}, we need the following technical Lemma.

\begin{lem} \label{l34}
Let the circle act on a $2n$-dimensional compact symplectic manifold $(M,\omega)$. Let $v$ and $w$ be fixed points in the same $2$-dimensional component $Z$ of $M^{\mathbb{Z}_d}$, where $d$ is the largest weight. Also suppose that $-d \in \Sigma_{v}$, $d \in \Sigma_{w}$, and $c_{1}(M)|_{v}=c_{1}(M)|_{w}$. Then $\lambda_{v}+2= \lambda_{w}$.
\end{lem}

\begin{proof} By Lemma \ref{l25}, $\Sigma_{v} \equiv \Sigma_{w} \mod d$. Let $\xi_{v}^{i}, 1 \leq i \leq n$, and $\xi_{w}^{i}, 1 \leq i \leq n$, be the weights at $v$ and $w$, respectively, where $\xi_{v}^{i},\xi_{w}^{i} \in \mathbb{Z} \setminus \{0\}$. By permuting if necessary, we can assume that $\xi_{v}^{i} \equiv \xi_{w}^{i} \mod d$, for all $i<n$, $\xi_{v}^{n}=-d$, and $\xi_{w}^{n}=d$. By Lemma \ref{l24},
$d > |\xi_{v}^{i}|$ and $d > |\xi_{w}^{i}|$, for $i<n$.
Then for all $i<n$, the following holds:
\begin{enumerate}
\item If $\xi_{v}^{i}>0$ and $\xi_{w}^{i}>0$, or if $\xi_{v}^{i}<0$ and $\xi_{w}^{i}<0$, then $\xi_{v}^{i} \equiv \xi_{w}^{i} \mod d$ implies $\xi_{v}^{i}-\xi_{w}^{i}=0$.
\item If $\xi_{v}^{i}>0$ and $\xi_{w}^{i}<0$, then $\xi_{v}^{i} \equiv \xi_{w}^{i} \mod d$ implies $\xi_{v}^{i}-\xi_{w}^{i}=d$.
\item If $\xi_{v}^{i}<0$ and $\xi_{w}^{i}>0$, then $\xi_{v}^{i} \equiv \xi_{w}^{i} \mod d$ implies $\xi_{v}^{i}-\xi_{w}^{i}=-d$.
\end{enumerate}

Moreover, there are $\frac{\lambda_{v}}{2}-1$ negative weights in $\Sigma_{v}$ excluding $-d$ and $\frac{\lambda_{w}}{2}$ negative weights in $\Sigma_{w}$. Hence,
\begin{center}
$0=c_{1}(M)|_{v}-c_{1}(M)|_{w}$ $=(\xi_{v}^{1}+\cdots+\xi_{v}^{n-1}-d)-(\xi_{w}^{1}+\cdots+\xi_{w}^{n-1}+d)$
$=(\xi_{v}^{1}-\xi_{w}^{1})+\cdots+(\xi_{v}^{n-1}-\xi_{w}^{n-1})-2d$
$\D{=d\bigg( \frac{\lambda_{w}}{2}-\frac{\lambda_{v}}{2}+1 \bigg) -2d}$
$\D{=d\bigg( \frac{\lambda_{w}}{2}-\frac{\lambda_{v}}{2}-1 \bigg) .}$
\end{center}
Therefore, $\lambda_{v}+2= \lambda_{w}$.
\end{proof}

\begin{rem} \label{r35}
We can generalize Lemma \ref{l34} in the following way: Let the circle act on a $2n$-dimensional compact symplectic manifold $(M,\omega)$. Let $v$ and $w$ be fixed points in the same component $Z$ of $M^{\mathbb{Z}_d}$, where $d$ is the largest weight. Then 
\begin{center}
$\D{\lambda_{v}(M) - \lambda_{w}(M) + \lambda_{v}(Z) - \lambda_{w}(Z) = -\frac{2}{d}\bigg(c_{1}(M)|_{v}-c_{1}(M)|_{w}\bigg).}$
\end{center}
The proof goes similarly to that of Lemma \ref{l34}.
\end{rem}

Finally, when the largest weight is odd, we will need the following closely related technical lemma.

\begin{lem} \label{l36}
Let the circle act on a $2n$-dimensional compact symplectic manifold $(M,\omega)$. Suppose that fixed points $v$ and $w$ satisfy the conditions in Lemma \ref{l34}. Let $d$ be the largest weight and assume that $d$ is odd. Suppose that $\Sigma_{v}$ and $\Sigma_{w}$ have $E_{v}^{+}$ and $E_{w}^{+}$ positive even weights and $E_{v}^{-}$ and $E_{w}^{-}$ negative even weights, respectively. Then $E_{v}^{+}-E_{v}^{-}-E_{w}^{+}+E_{w}^{-}=2$.
\end{lem}

\begin{proof} By Lemma \ref{l25}, $\Sigma_{v} \equiv \Sigma_{w} \mod d$. Define $\xi_{v}^{i}$'s and $\xi_{w}^{i}$'s as in Lemma \ref{l34} and recall that the following hold:
\begin{enumerate}[(a)]
\item If $\xi_{v}^{i}>0$ and $\xi_{w}^{i}>0$, or if $\xi_{v}^{i}<0$ and $\xi_{w}^{i}<0$, then $\xi_{v}^{i}-\xi_{w}^{i}=0$.
\item If $\xi_{v}^{i}>0$ and $\xi_{w}^{i}<0$, then $\xi_{v}^{i}-\xi_{w}^{i}=d$.
\item If $\xi_{v}^{i}<0$ and $\xi_{w}^{i}>0$, then $\xi_{v}^{i}-\xi_{w}^{i}=-d$.
\end{enumerate}

Let $e_{v}^{+}$ be a positive even weight at $v$, $e_{v}^{-}$ a negative even weight at $v$, $o_{v}^{+}$ a positive odd weight at $v$, and $o_{v}^{-}$ a negative odd weight at $v$, and similarly for $w$. Then since the largest weight $d$ is odd, we have the following:
\begin{enumerate}
\item $e_{v}^{+} \equiv e_{w}^{+} \mod d$ implies that $e_{v}^{+} = e_{w}^{+}$. Hence in $c_{1}(M)|_{v}-c_{1}(M)|_{w}$, this pair contributes $0$. Suppose that there are $k_{1}$ such pairs.
\item $e_{v}^{+} \equiv o_{w}^{-} \mod d$ implies that $e_{v}^{+} - o_{w}^{-} =d$. Hence in $c_{1}(M)|_{v}-c_{1}(M)|_{w}$, this pair contributes $d$. There are $E_{v}^{+}-k_{1}$ such pairs.
\item $e_{v}^{-} \equiv e_{w}^{-} \mod d$ implies that $e_{v}^{-} = e_{w}^{-}$. Hence in $c_{1}(M)|_{v}-c_{1}(M)|_{w}$, this pair contributes as $0$. Suppose that there are $k_{2}$ such pairs.
\item $e_{v}^{-} \equiv o_{w}^{+} \mod d$ implies that $e_{v}^{-} - o_{w}^{+} = -d$. Hence in $c_{1}(M)|_{v}-c_{1}(M)|_{w}$, this pair contributes $-d$. There are $E_{v}^{-}-k_{2}$ such pairs.
\item $o_{v}^{+} \equiv o_{w}^{+} \mod d$ implies that $o_{v}^{+} = o_{w}^{+}$. Hence in $c_{1}(M)|_{v}-c_{1}(M)|_{w}$, this pair contributes $0$. Suppose that there are $k_{3}$ such pairs.
\item $o_{v}^{+} \equiv e_{w}^{-} \mod d$ implies that $o_{v}^{+} - e_{w}^{-} = d$. Hence in $c_{1}(M)|_{v}-c_{1}(M)|_{w}$, this pair contributes $d$. There are $E_{w}^{-}-k_{2}$ such pairs.
\item $o_{v}^{-} \equiv e_{w}^{+} \mod d$ implies that $o_{v}^{-} - e_{w}^{+} = -d$. Hence in $c_{1}(M)|_{v}-c_{1}(M)|_{w}$, this pair contributes $-d$. There are $E_{w}^{+}-k_{1}$ such pairs.
\item $o_{v}^{-} \equiv o_{w}^{-} \mod d$ implies that $o_{v}^{-} = o_{w}^{-}$. Hence in $c_{1}(M)|_{v}-c_{1}(M)|_{w}$, this pair contributes $0$. Suppose that there are $k_{4}$ such pairs.
\end{enumerate}
Then
\begin{center}
$0=c_{1}(M)|_{v}-c_{1}(M)|_{w}$
$=d(E_{v}^{+}-k_{1})-d(E_{v}^{-}-k_{2})+d(E_{w}^{-}-k_{2})-d(E_{w}^{+}-k_{1})-2d$
$=d(E_{v}^{+}-E_{v}^{-}-E_{w}^{+}+E_{w}^{-}-2).$
\end{center}
\end{proof}

\begin{rem} \label{r37}
We can also generalize Lemma \ref{l36} in the following way: Let the circle act on a $2n$-dimensional compact symplectic manifold $(M,\omega)$. Let $v$ and $w$ be fixed points in the same component $Z$ of $M^{\mathbb{Z}_d}$, where $d$ is the largest weight. Assume that the largest weight $d$ is odd. Then 
\begin{center}
$d(E_{v}^{+}-E_{v}^{-}-E_{w}^{+}+E_{w}^{-})=c_{1}(M)|_{v}-c_{1}(M)|_{w}-c_{1}(Z)|_{v}+c_{1}(Z)|_{w}.$
\end{center}
The proof goes similarly to that of Lemma \ref{l36}.
\end{rem}

\section{Proof of Theorem~\ref{t11}}

The proof of Theorem \ref{t11} is based on induction. The key fact to prove Theorem \ref{t11} is that an isotropy submanifold of a symplectic manifold is itself a smaller symplectic manifold.

To prove the base case, we need several theorems:

\begin{pro} \label{p41} \cite{MD}
An effective symplectic circle action on a 4-dimensional compact, connected symplectic manifold is Hamiltonian if and only if the fixed point set if non-empty.
\end{pro}

Let the circle act symplectically on a 4-dimensional compact, connected symplectic manifold $M$ with isolated fixed points. Assume that the action is effective. Then we can associate a graph to $M$ in the following way: We assign a vertex to each fixed point. Label each fixed point by its moment image. Additionally, given two fixed points $p$ and $q$, we say that $(p,q)$ is an edge if there exists $k>1$ such that $p$ and $q$ are contained in the same component of the isotropy submanifold $M^{\mathbb{Z}_k}$, where $k$ is the largest such. We label the edge by $k$.

\begin{theo} (Uniqueness Theorem) \label{t42} \cite{K}
Let $(M, \omega, \Pi)$ and $(M', \omega', \Pi')$ be two compact four dimensional Hamiltonian $S^{1}$ spaces. Then any isomorphism between their corresponding graphs is induced by an equivariant symplectomorphism.
\end{theo}

We now prove the base case.

\begin{pro} \label{p43}
Let the circle act symplectically on a compact, connected symplectic manifold $M$ and suppose that there are exactly three fixed points. If $\dim M < 8$, then $M$ is equivariantly symplectomorphic to $\mathbb{CP}^2$.
\end{pro}

\begin{proof} Suppose that $\dim M < 8$. By quotienting out by the subgroup which acts trivially, we may assume that the action is effective. Since the manifold $M$ is connected, $\dim M \neq 0$. Hence by Corollary \ref{c23}, $\dim M=4$. Let $p, q$, and $r$ denote the three fixed points and without loss of generality assume that $\lambda_{p} \leq \lambda_{q} \leq \lambda_{r}$. Then by Lemma \ref{l34}, $\lambda_{p}=0, \lambda_{q}=2$, and $\lambda_{r}=4$.

By a standard action on $\mathbb{CP}^{2}$ we mean that for each $\lambda \in S^1$, $\lambda$ acts on $\mathbb{CP}^{2}$ by 
\begin{center}
$\lambda \cdot [z_0 : z_1 :z_2 ]= [ \lambda^a z_0 : \lambda^b z_1 : z_2]$
\end{center}
for some natural numbers $a$ and $b$. This action has three fixed points $[1:0:0]$, $[0:1:0]$, and $[0:0:1]$. And the weights at these points are $\{a,a+b\}$, $\{-a,b\}$, and $\{-b,-a-b\}$.

Since $\dim M=4$, by Proposition \ref{p41}, the action is Hamiltonian. Furthermore, by Lemma \ref{l24}, there exist natural numbers $a$, $b$, and $c$ such that the weights are $\Sigma_p=\{a,c\}$, $\Sigma_q=\{-a,b\}$, and $\Sigma_r=\{-b,-c\}$. By Theorem \ref{t21},
\begin{center}
$\D{0=\int_M 1=\frac{1}{ac}-\frac{1}{ab}+\frac{1}{bc}=\frac{b-c+a}{abc}.}$
\end{center}
Thus $c=a+b$. It is straightforward to check that the corresponding graph is isomorphic to a graph corresponding to some standard action on $\mathbb{CP}^2$ where the action is given by $\lambda \cdot [z_0 : z_1 :z_2 ]= [ \lambda^a z_0 : \lambda^b z_1 : z_2]$, and hence this induces an equivariant symplectomorphism on manifolds by Theorem \ref{t42}.
\end{proof}

Hence from now on we assume that $\dim M \geq 8$. Then note that, by Corollary \ref{c27}, the Chern class map is identically zero.

\begin{lem} \label{l44} 
Fix a natural number $n$ such that $n \geq 4$. Assume that Theorem \ref{t11} holds for all manifolds $M$ such that $\dim M < 2n$. Let the circle act symplectically on a $2n$-dimensional compact, connected symplectic manifold $M$ and suppose that there are exactly three fixed points. Assume that the action is effective. Then there exist even natural numbers $a$ and $b$ such that the weights at the three fixed points in the isotropy submanifold $M^{\mathbb{Z}_{2}}$ are $\{a, c\}, \{-a, b\}$, and $\{-b,-c\}$, where $c=a+b$.
\end{lem}

\begin{proof}
Since the action is effective, the isotropy submanifold $\mathbb{Z}_{2}$ is a smaller manifold, i.e., for any component $Z$ of  $M^{\mathbb{Z}_{2}}$, we have that $\dim Z < \dim M$. Then by the inductive hypothesis and Theorem \ref{t26}, there are only four possible cases:
\begin{enumerate}
\item Each fixed point is a component of the isotropy submanifold $M^{\mathbb{Z}_2}$.
\item The isotropy submanifold $M^{\mathbb{Z}_2}$ contains a 2-sphere with two fixed points. The third fixed point is another component of $M^{\mathbb{Z}_2}$.
\item The isotropy submanifold $M^{\mathbb{Z}_2}$ contains a 4-dimensional component with the three fixed points.
\item The isotropy submanifold $M^{\mathbb{Z}_2}$ contains a 6-dimensional component with two fixed points. The third fixed point is another component of $M^{\mathbb{Z}_2}$.
\end{enumerate}

Assume that the first case holds. Let $p,q$, and $r$ be the fixed points. The first case means that all the weights at $p,q$, and $r$ are odd. Let $A, B$, and $C$ be the products of the weights at $p, q,$ and $r$, respectively. Then by Theorem \ref{t21},
\begin{center}
$\D{\int_{M} 1 = \sum_{F \subset M^{S^{1}}} \int_{F} \frac{1}{e_{S^{1}(N_{F})}} = \frac{1}{\prod \xi_{p}^{i}}t^{-n} + \frac{1}{\prod \xi_{q}^{i}}t^{-n} + \frac{1}{\prod \xi_{r}^{i}}t^{-n}}$
$\D{= \bigg(\frac{1}{A}+\frac{1}{B}+\frac{1}{C} \bigg)t^{-n}=0.}$
\end{center}
Hence 
\begin{center}
$\D{\frac{1}{A}+\frac{1}{B}+\frac{1}{C}=0.}$
\end{center}
Multiplying both sides by $ABC$ yields
\begin{center}
$BC+AC+AB =0.$
\end{center}
However, since $A, B,$ and $C$ are odd,
\begin{center}
$BC+AC+AB \equiv 1 \mod 2,$
\end{center}
which is a contradiction.

Assume that the second case holds. Then the two fixed points in the 2-sphere have one even weight and $n-1$ odd weights. By Corollary \ref{c23}, $n$ is even. Then sums of the weights at these points are congruent to $1 \mod 2$, which contradicts Corollary \ref{c27} that the first Chern class map (the sum of the weights at a fixed point) is zero for all fixed points if $\dim M \geq 8$ and there are exactly three fixed points.

Assume that the fourth case holds. Then the two fixed points in the 2-sphere have three even weights and $n-3$ odd weights. By Corollary \ref{c23}, $n$ is even. Again, the sums of weights at these points are congruent to $1 \mod 2$, which contradicts that the first Chern class map is zero for all fixed points by Corollary \ref{c27}.

Hence the third case is the case. Thus as in the proof of Proposition \ref{p43}, there are even natural numbers $a$ and $b$ such that the fixed points of the isotropy submanifold $M^{\mathbb{Z}_2}$ have weights $\{a+b, a\}, \{-a, b\},$ and $\{-b,-a-b\}$.
\end{proof}

\begin{lem} \label{l45}
Fix a natural number $n$. Assume that Theorem \ref{t11} holds for all manifolds $M$ such that $\dim M < 2n$. Let the circle act symplectically on a $2n$-dimensional compact, connected symplectic manifold $M$ and suppose that there are exactly three fixed points. Assume that the action is effective. Given an integer $e \in \mathbb{Z}\setminus\{-1,0,1\}$, exactly one of the following holds:
\begin{enumerate}
\item Each fixed point is a component of the isotropy submanifold $M^{\mathbb{Z}_e}$.
\item The isotropy submanifold $M^{\mathbb{Z}_e}$ contains a 2-sphere with two fixed points, and the weights in the 2-sphere at these points are $\{a\}$ and $\{-a\}$ for some natural number $a$ that is a multiple of $e$. The third fixed point is another component of $M^{\mathbb{Z}_e}$.
\item The isotropy submanifold $M^{\mathbb{Z}_e}$ contains a 4-dimensional component with the three fixed points, and the weights in the isotropy submanifold at these points are $\{a+b,a\}$, $\{-a,b\}$, and $\{-b,-a-b\}$ for some natural numbers $a$ and $b$ that are multiples of $e$.
\item The isotropy submanifold $M^{\mathbb{Z}_e}$ contains a 6-dimensional component with two fixed points, and the weights in the isotropy submanifold at these points are $\{a,b,-a-b\}$ and $\{a+b,-a,-b\}$ for some natural numbers $a$ and $b$ that are multiples of $e$. The third fixed point is another component of $M^{\mathbb{Z}_e}$.
\end{enumerate}
\end{lem}

\begin{proof} Fix an integer $e \in \mathbb{Z}\setminus\{-1,0,1\}$. Since the action on $M$ is effective, for any component $Z$ of the isotropy submanifold $M^{\mathbb{Z}_e}$, we have that $\dim Z < \dim M$.

By ABBV Localization (Theorem \ref{t21}), if any component of the isotropy submanifold $M^{\mathbb{Z}_e}$ only contains one fixed point, then the fixed point itself is the component.

If every fixed point is itself a component of the isotropy submanifold $M^{\mathbb{Z}_e}$, this is the first case of the Lemma.

Suppose instead that there exists a component $Z$ of the isotropy submanifold $M^{\mathbb{Z}_e}$ that contains exactly two fixed points. Then by Theorem \ref{t26}, either
\begin{enumerate}[(a)]
\item The component is 2-sphere and the weights in the isotropy submanifold $M^{\mathbb{Z}_e}$ at these points are $\{a\}$ and $\{-a\}$ for some natural number $a$ that is a multiple of $e$. By the previous argument, the third fixed point is another component of $M^{\mathbb{Z}_e}$. This is the second case of the Lemma.
\item The component is 6-dimensional and the weights in the isotropy submanifold $M^{\mathbb{Z}_e}$ at these points are $\{a,b,-a-b\}$ and $\{a+b,-a,-b\}$ for some natural numbers $a$ and $b$ that are multiples of $e$. The third fixed point is another component of $M^{\mathbb{Z}_e}$. This is the fourth case of the Lemma. 
\end{enumerate}

Finally, suppose that a component of the isotropy submanifold $M^{\mathbb{Z}_e}$ contains the three fixed points. Then by the inductive hypothesis, the component is 4-dimensional and the weights in the isotropy submanifold are $\{a+b,a\}$, $\{-a,b\}$, and $\{-b,-a-b\}$ for some natural numbers $a$ and $b$ that are multiples of $e$. This is the third case of the Lemma.
\end{proof}

As particular cases of Lemma \ref{l45}, we need the following Lemma.

\begin{lem} \label{l46}
Fix a natural number $n$. Assume that Theorem \ref{t11} holds for all manifolds $M$ such that $\dim M < 2n$. Let the circle act symplectically on a $2n$-dimensional compact, connected symplectic manifold $M$ and suppose that there are exactly three fixed points. Assume that the action is effective. Fix an integer $e \in \mathbb{Z}\setminus\{-1,0,1\}$.
\begin{enumerate}
\item
Suppose that there exist distinct fixed points $\alpha$ and $\beta$ such that $N_{\alpha}(e)>0$ and $N_{\beta}(-e)>0$ such that $|e|>\frac{d}{2}$ where $d$ is the largest weight. Then $N_{\alpha}(e)=1$, $N_{\beta}(-e)=1$, $\Sigma_{\alpha} \equiv \Sigma_{\beta} \mod e$, and no additional multiples of $e$ appear as weights.
\item
If there exist two distinct fixed points $\alpha$ and $\beta$ such that $N_{\alpha}(e)>0$ and $N_{\beta}(e)>0$, then after possibly switching $\alpha$ and $\beta$,
\begin{center}
$\{2e,e\} \subset \Sigma_{\alpha}$, $\{-e,e\} \subset \Sigma_{\beta}$, and $\{-2e,-e\} \subset \Sigma_{\gamma}$
\end{center}
where $\gamma$ is the remaining fixed point. Moreover, no additional multiples of $e$ appear as weights. 
\item
If there exists a fixed point $\alpha$ such that $N_{\alpha}(e)>1$, $\{-2e,e,e\} \subset \Sigma_{\alpha}$ and $\{2e,-e,-e\} \subset \Sigma_{\beta}$ for some fixed point $\beta \neq \alpha$. Moreover, no additional multiples of $e$ appear as weights.
\item
Suppose that there exists a fixed point $\beta$ such that $N_{\beta}(e)>0$ and $N_{\beta}(-e)>0$. Then
\begin{center}
$\{2e,e\} \subset \Sigma_{\alpha}$, $\{-e,e\} \subset \Sigma_{\beta}$, and $\{-2e,-e\} \subset \Sigma_{\gamma}$
\end{center}
where $\alpha$ and $\gamma$ are the remaining two fixed points. Moreover, no additional multiples of $e$ appear as weights.
\end{enumerate}
\end{lem}

\begin{proof}
Fix an integer $e \in \mathbb{Z}\setminus\{-1,0,1\}$ and consider the isotropy submanifold $M^{\mathbb{Z}_e}$.
\begin{enumerate}
\item By looking at the weights in the isotropy submanifold $M^{\mathbb{Z}_e}$, the second, the third, and the fourth cases are possible. In the third case or the fourth case, $a \geq |e|$ and $b \geq |e|$ hence $a+b \geq 2|e| > d$, which is a contradiction. Hence this must be the second case of Lemma \ref{l45} with $a=|e|$. Moreover, $\alpha$ and $\beta$ lie in the same 2-sphere of $M^{\mathbb{Z}_e}$. Hence by Lemma \ref{l25}, $\Sigma_{\alpha} \equiv \Sigma_{\beta} \mod e$.
\item By looking at the weights in the isotropy submanifold $M^{\mathbb{Z}_e}$, this must be the third case of Lemma \ref{l45} with $a=b=|e|$.
\item By looking at the weights in the isotropy submanifold $M^{\mathbb{Z}_e}$, this must be the fourth case of Lemma \ref{l45} with $a=b=|e|$.
\item By looking at the weights in the isotropy submanifold $M^{\mathbb{Z}_e}$, this must be the third case of Lemma \ref{l45} with $a=b=|e|$.
\end{enumerate}
\end{proof}

\section{The largest weight odd case}

Let the circle act symplectically on a compact, connected symplectic manifold $M$ with exactly three fixed points. In this section, we show that if $\dim M \geq 8$, the largest weight cannot be odd.

\begin{pro} \label{p51}
Fix a natural number $n$ such that $n \geq 4$. Assume that Theorem \ref{t11} holds for all manifolds $M$ such that $\dim M < 2n$. Let the circle act symplectically on a $2n$-dimensional compact, connected symplectic manifold $M$ and suppose that there are exactly three fixed points. Assume that the action is effective. Then the largest weight is even.
\end{pro}

\begin{proof}
Assume on the contrary that the largest weight is odd. Then this is an immediate consequence of Lemma \ref{l52}, Lemma \ref{l56}, and Lemma \ref{l57} below.
\end{proof}

\begin{lem} \label{l52}
Fix a natural number $n$ such that $n \geq 4$. Assume that Theorem \ref{t11} holds for all manifolds $M$ such that $\dim M < 2n$. Let the circle act symplectically on a $2n$-dimensional compact, connected symplectic manifold $M$ and suppose that there are exactly three fixed points $p,q$, and $r$, with $\lambda_p \leq \lambda_q \leq \lambda_r$. Assume that the action is effective and the largest weight $d$ is odd. Then after possibly reversing the circle action we may assume that $-d \in \Sigma_p$ and $d \in \Sigma_q$, and there exist even natural numbers $a$ and $b$ such that either
\begin{enumerate}
\item $\{a,c\} \subset \Sigma_{p}, \{-a,b\} \subset \Sigma_{q},$ and $ \{-b,-c\} \subset \Sigma_{r}$; or
\item $\{-a,b\} \subset \Sigma_{p}, \{-b,-c\} \subset \Sigma_{q}$, and $ \{a,c\} \subset \Sigma_{r}$,
\end{enumerate}
where $c=a+b$. Moreover, these are the only even weights.
\end{lem}

\begin{proof}
By Lemma \ref{l32}, $N_{p}(d)+N_{q}(d)+N_{r}(d)=1$ and $N_{p}(-d)+N_{q}(-d)+N_{r}(-d)=1$. By Lemma \ref{l33}, $\lambda_{p}=n-2, \lambda_{q}=n$, and $\lambda_{r}=n+2$. Moreover, after possibly reversing the circle action, we may assume that $-d \in \Sigma_p$ and $d \in \Sigma_q$.

By Lemma \ref{l44}, there exist even natural numbers $a$ and $b$ such that the weights at the three fixed points in the isotropy submanifold $M^{\mathbb{Z}_{2}}$ are $\{a, c\}, \{-a, b\}$, and $\{-b,-c\}$, where $c=a+b$. In the Lemma, the order is not specified. We have six possible cases. Other four cases are:
\begin{enumerate}[a.]
\item $\{a,c\} \subset \Sigma_{p}, \{-b,-c\} \subset \Sigma_{q}$, and $ \{-a,b\} \subset \Sigma_{r}$.
\item $\{-a,b\} \subset \Sigma_{p}, \{a,c\} \subset \Sigma_{q},$ and $ \{-b,-c\} \subset \Sigma_{r}$.
\item $\{-b,-c\} \subset \Sigma_{p}, \{-a,b\} \subset \Sigma_{q},$ and $ \{a,c\} \subset \Sigma_{r}$.
\item $\{-b,-c\} \subset \Sigma_{p}, \{a,c\} \subset \Sigma_{q},$ and $ \{-a,b\} \subset \Sigma_{r}$.
\end{enumerate}
The fixed points $p$ and $q$ satisfy the conditions in Lemma \ref{l36}. Therefore, $E_{v}^{+}-E_{v}^{-}-E_{w}^{+}+E_{w}^{-}=2$ and the other cases are ruled out.
\end{proof}

\begin{lem} \label{l53} Under the assumption of Lemma \ref{l52}, $a \neq b$.
\end{lem}

\begin{proof} Assume on the contrary that $a = b$. Since $-d \in \Sigma_p$ and $d \in \Sigma_q$ where $d$ is the largest weight, by Lemma \ref{l46} part 1 for $d$, $\Sigma_{p} \equiv \Sigma_{q} \mod d$. As a result, we can find a bijection between the weights at $p$ and the weights at $q$ that takes each weight $\alpha$ at $p$ to a weight $\beta$ at $q$ such that $\alpha \equiv \beta \mod d$. Moreover, since $a=b$, we can take this bijection to take $a$ at $p$ to $b$ at $q$ in the first case, and we can take this bijection to take $-a$ at $p$ to $-b$ at $q$ in the second case.

Assume that the first case in Lemma \ref{l52} holds, i.e., $\{a,c\} \subset \Sigma_{p}, \{-a,b\} \subset \Sigma_{q},$ and $ \{-b,-c\} \subset \Sigma_{r}$. Moreover, these are the only even weights. First, $-d$ at $p$ has to go to $d$ at $q$ since all the other weights are non-zero and have absolute values less than $d$. Next, $c$ at $p$ must go to $c-d$ at $q$ and $-a$ at $q$ must go to $d-a$ at $p$. If $l$ is any remaining positive odd weight at $p$, then it has to go to $l$ at $q$ since the largest weight $d$ is odd. Similarly, any negative odd weight $-k$ at $p$ must go to $-k$ at $q$.

By Corollary \ref{c23}, $\frac{1}{2}\dim M$ is even. Since $\lambda_p=\frac{1}{2}\dim M-2$ and $\lambda_q=\frac{1}{2}\dim M$ by Lemma \ref{l33}, this implies that the weights at $p$ and $q$ are
\begin{center}
$\Sigma_{p}=\{-d, a, c, d-a\} \cup \{x_{i}\}_{i=1}^{t} \cup \{-y_{i}\}_{i=1}^{t}$
$\Sigma_{q}=\{d, b, c-d, -a\} \cup \{x_{i}\}_{i=1}^{t} \cup \{-y_{i}\}_{i=1}^{t}$
\end{center}
for some odd natural numbers $x_{i}$'s and $y_{i}$'s where $\dim M=8+4t$, for some $t \geq 0$.

Suppose that $x_{i}>1$ for some $i$. Then by Lemma \ref{l46} part 2, $\{2x_i,x_i\}\subset\Sigma_p$, $\{-x_i,x_i\}\subset\Sigma_q$, and $\{-2x_i,-x_i\}\subset\Sigma_r$, or $\{-x_i,x_i\}\subset\Sigma_p$, $\{2x_i,x_i\}\subset\Sigma_q$, and $\{-2x_i,-x_i\}\subset\Sigma_r$. Moreover, no more multiples of $x_i$ should sppear as weights. This implies that $-x_i \neq -y_j$ for all $j$. Since $-y_j$'s are the only negative odd weights at $p$, this implies that the second case is impossible. Assume that the first case holds. Then we must have $c=2x_i$. Also, since $\{-x_i,x_i\}\subset\Sigma_q$ but $-x_i \neq -y_j$ for all $j$, $-x_{i}=c-d$. However, this means that $2x_{i}-d=c-d=-x_{i}$ hence $d=3x_{i}$, which is a contradiction since no more multiples of $x_{i}$ should appear.

Hence $x_{i}=1$ for all $i$. Similarly, one can show that $y_{i}=1$ for all $i$. Then $c_{1}(M)|_{p}=-d+a+c+d-a=c>0$, which is a contradiction by Corollary \ref{c27} that the first Chern class map is identically zero.

Similarly, we get a contradiction of the second case of Lemma \ref{l52} with $a=b$ by a slight variation of this argument.
\end{proof}

\begin{lem} \label{l54} Assume that the first case in Lemma \ref{l52} holds. Then the weights are
$$\Sigma_{p}=\{-d, a, c, d-a, b-d, 1\} \cup \{-1, 1\}_{i=1}^{t}$$
$$\Sigma_{q}=\{d, a-d, c-d, -a, b, 1\} \cup \{-1, 1\}_{i=1}^{t}$$
$$\Sigma_{r}=\{-b, -c, \cdots \}$$
where the largest weight $d$ is odd, $a,b$, and $c$ are even natural numbers such that $c=a+b$, and $\dim M=12+4t$ for some $t\geq 0$. Moreover, $a \neq b$ and the remaining weights at $r$ are odd.
\end{lem}

\begin{proof}
Assume that the first case in Lemma \ref{l52} holds, i.e., there exist even natural numbers $a$, $b$, and $c$ such that $\{a,c\} \subset \Sigma_{p}, \{-a,b\} \subset \Sigma_{q},$ and $ \{-b,-c\} \subset \Sigma_{r}$ where $c=a+b$. Moreover, these are the only even weights.

Since $-d \in \Sigma_p$ and $d \in \Sigma_q$ where $d$ is the largest weight, by Lemma \ref{l46} part 1 for $d$, $\Sigma_{p} \equiv \Sigma_{q} \mod d$. First $-d$ at $p$ has to go to $d$ at $q$ since all the other weights are non-zero and have absolute values less than $d$. Second, by Lemma \ref{l53}, $a \neq b$. Hence $a$ at $p$ must go to $a-d$ at $q$ and $b$ at $q$ must go to $b-d$ at $p$. Next, $c$ at $p$ must go to $c-d$ at $q$ and $-a$ at $q$ must go to $d-a$ at $p$. If $l$ is any remaining positive odd weight at $p$, then it has to go to $l$ at $q$ since the largest weight $d$ is odd. Similarly, any remaining negative odd weight $-k$ at $p$ must go to $-k$ at $q$.

By Corollary \ref{c23}, $\frac{1}{2}\dim M$ is even. Since $\lambda_p=\frac{1}{2}\dim M-2$ and $\lambda_q=\frac{1}{2}\dim M$ by Lemma \ref{l33}, this implies that the weights at $p$ and $q$ are
\begin{center}
$\Sigma_{p}=\{-d, a, c, d-a, b-d\} \cup \{x_{i}\}_{i=1}^{t+1} \cup \{-y_{i}\}_{i=1}^{t}$
$\Sigma_{q}=\{d, a-d, c-d, -a, b\} \cup \{x_{i}\}_{i=1}^{t+1} \cup \{-y_{i}\}_{i=1}^{t}$
\end{center}
for some odd natural numbers $x_{i}$'s and $y_{i}$'s where $\dim M=12+4t$, for some $t \geq 0$. We also have
\begin{center}
$\Sigma_{r}=\{-b, -c, \cdots \}.$
\end{center}
We show that $x_{i}=y_{i}=1$ for all $i$.

\begin{enumerate}
\item $x_{i}=1$ for all $i$.

Suppose not. Without loss of generality, assume that $x_1>1$. Then by Lemma \ref{l46} part 2 for $x_1$, $\{2x_1,x_1\}\subset\Sigma_p$, $\{-x_1,x_1\}\subset\Sigma_q$, and $\{-2x_i,-x_i\}\subset\Sigma_r$, or $\{-x_1,x_1\}\subset\Sigma_p$, $\{2x_1,x_1\}\subset\Sigma_q$, and $\{-2x_i,-x_i\}\subset\Sigma_r$. Moreover, no more multiples of $x_1$ should sppear as weights. This implies that $-x_1 \neq -y_j$ for all $j$. If the first case holds, we must have that $c=2x_1$. Also, there must be a weight at $q$ that is equal to $-x_1$. If $-x_1=a-d$, $\{2x_1,x_1,x_1\}=\{c,-a+d,x_1\} \subset \Sigma_{p}$, which is not possible by Lemma \ref{l46}. If $-x_1=c-d$, $-x_1=c-d=2x_1-d$ implies that $d=3x_1$, which contradicts that no more multiples of $x_1$ should appear as weights. If the second case holds, we must have that $b=2x_1$. Also, there must be a weight at $p$ that is equal to $-x_1$. Since $-x_1 \neq -d$, $-x_1=b-d$. However, $-x_1=b-d=2x_1-d$ implies that $d=3x_1$, which contradicts that no more multiples of $x_1$ should appear as weights.
 
\item $y_{i}=1$ for all $i$.

Suppose not. Without loss of generality, assume that $y_{1}>1$. Then by Lemma \ref{l46} part 2 for $y_1$, we must have that $\{2y_{1}, y_{1}\}\subset\Sigma_{r}$, which is a contradiction since $r$ has no positive even weight.
\end{enumerate}
\end{proof}

\begin{lem} \label{l55} Assume that the second case in Lemma \ref{l52} holds. Then the weights are
$$\Sigma_{p}=\{-d, -a, b, d-b, d-c, 1\} \cup \{-1,1\}_{i=1}^{t}$$
$$\Sigma_{q}=\{d, -b, -c, d-a, b-d, 1\} \cup \{-1,1\}_{i=1}^{t}$$
$$\Sigma_{r}=\{a, c, \cdots \},$$
where the largest weight $d$ is odd, $a,b$, and $c$ are even natural numbers such that $c=a+b$, and $\dim M=12+4t$ for some $t\geq 0$. Moreover, $a \neq b$ and the remaining weights at $r$ are odd.
\end{lem}

\begin{proof}
Assume that the second case in Lemma \ref{l52} holds, i.e., there exist even natural numbers $a$, $b$, and $c$ such that $\{-a,b\} \subset \Sigma_{p}, \{-b,-c\} \subset \Sigma_{q},$ and $ \{a,c\} \subset \Sigma_{r}$ where $c=a+b$. Moreover, these are the only even weights.

Since $-d \in \Sigma_p$ and $d \in \Sigma_q$ where $d$ is the largest weight, by Lemma \ref{l46} part 1 for $d$, $\Sigma_{p} \equiv \Sigma_{q} \mod d$. First $-d$ at $p$ has to go to $d$ at $q$ since all the other weights are non-zero and have absolute values less than $d$. Second, by Lemma \ref{l53}, $a \neq b$. Hence $-a$ at $p$ must go to $d-a$ at $q$ and $-b$ at $q$ must go to $d-b$ at $p$. Next, $c$ at $p$ must go to $c-d$ at $q$ and $-a$ at $q$ must go to $d-a$ at $p$. If $l$ is any remaining positive odd weight at $p$, then it has to go to $l$ at $q$ since the largest weight $d$ is odd. Similarly, any negative odd weight $-k$ at $p$ must go to $-k$ at $q$.

By Corollary \ref{c23}, $\frac{1}{2}\dim M$ is even. Since $\lambda_p=\frac{1}{2}\dim M-2$ and $\lambda_q=\frac{1}{2}\dim M$ by Lemma \ref{l33}, this implies that the weights at $p$ and $q$ are
\begin{center}
$\Sigma_{p}=\{-d, -a, b, d-b, d-c\} \cup \{x_{i}\}_{i=1}^{t+1} \cup \{-y_{i}\}_{i=1}^{t}$
$\Sigma_{q}=\{d, -b, -c, d-a, b-d\} \cup \{x_{i}\}_{i=1}^{t+1} \cup \{-y_{i}\}_{i=1}^{t}$
\end{center}
for some odd natural numbers $x_{i}$'s and $y_{i}$'s where $\dim M=12+4t$, for some $t \geq 0$. We also have
\begin{center}
$\Sigma_{r}=\{a, c, \cdots \}.$
\end{center}
We show that $x_{i}=y_{i}=1$ for all $i$.

\begin{enumerate}
\item $x_{i}=1$ for all $i$.

Suppose not. Without loss of generality, assume that $x_1>1$. Then by Lemma \ref{l46} part 2 for $x_1$, we must have that $\{-2x_1, -x_1\}\subset\Sigma_{r}$, which is a contradiction since $r$ has no negative even weight.

\item $y_{i}=1$, for all $i$.

Suppose not. Without loss of generality, assume that $y_{1}>1$. Then by Lemma \ref{l46} part 2 for $y_1$, $\{-2y_1,-y_1\}\subset\Sigma_p$, $\{-y_1,y_1\}\subset\Sigma_q$, and $\{2y_1,y_1\}\subset\Sigma_r$, or $\{-y_1,y_1\}\subset\Sigma_p$, $\{-2y_1,-y_1\}\subset\Sigma_q$, and $\{2y_1,y_1\}\subset\Sigma_r$. Moreover, no more multiples of $y_1$ should appear as weights.

If the first case holds, we must have that $a=2y_{1}$. Also, there must be a weight at $q$ that is equal to $y_1$. Thus, we have that $d-a=y_{1}$. However, this implies that $d-a=d-2y_{1}=y_{1}$ hence $d=3y_{1}$, which is a contradiction since no more multiples of $y_1$ should appear as weights.

Suppose that the second case holds. Then we must have that $c=2y_{1}$. Also, there must be a weight at $p$ that is equal to $y_1$. Hence $y_{1}=d-b$. Then $\{-2y_{1}, -y_{1}, -y_{1}\}=\{-c,b-d,-y_{2}\}\subset\Sigma_{q}$, which is a contradiction.
\end{enumerate}
\end{proof}

\begin{lem} \label{l56} The first case in Lemma \ref{l52} is not possible.
\end{lem}

\begin{proof}By Lemma \ref{l54}, the weights are
$$\Sigma_{p}=\{-d, a, c, d-a, b-d, 1\} \cup \{-1, 1\}_{i=1}^{t}$$
$$\Sigma_{q}=\{d, a-d, c-d, -a, b, 1\} \cup \{-1, 1\}_{i=1}^{t}$$
$$\Sigma_{r}=\{-b, -c, \cdots \},$$
where the largest weight $d$ is odd, $a,b$, and $c$ are even natural numbers such that $c=a+b$, and $\dim M=12+4t$ for some $t\geq 0$. Moreover, $a \neq b$ and the remaining weights at $r$ are odd.

We consider Lemma \ref{l24} for each integer. Lemma \ref{l24} holds for $d$, $a$, $b$, and $c$. Since $d>c>b$, $b-d < -1$. Since $a \neq b$, by Lemma \ref{l24} for $b-d$, it is straightforward to show that $d-b \in \Sigma_r$.

First, suppose that $c-d \neq -1$. Then $N_p(1)=N_p(-1)+1$ and $N_q(1)=N_q(-1)+1$. Hence by Lemma \ref{l24} for 1, $N_r(1)+2=N_r(-1)$. Considering Lemma \ref{l24} for each integer, one can show that the weights are
\begin{center}
$\Sigma_{p}=\{-d, a, c, d-a, b-d, 1\} \cup \{-1, 1\}_{i=1}^{t}$
$\Sigma_{q}=\{d, a-d, c-d, -a, b, 1\} \cup \{-1, 1\}_{i=1}^{t}$
$\Sigma_{r}=\{-b,-c,d-b,-1,d-c,-1\} \cup \{-e_{i},e_{i}\}_{i=1}^{t}$
\end{center}
for some odd natural numbers $e_{i}$'s. We show that $e_{i}=1$ for all $i$. 

Suppose that $e_{1}>1$. Then by Lemma \ref{l46} part 4, either $\{-2e_1,-e_1\}\subset\Sigma_{p}$, $\{2e_1,e_1\}\subset\Sigma_q$, and $\{-e_1,e_1\}\subset\Sigma_r$, or $\{2e_1,e_1\}\subset\Sigma_p$, $\{-2e_1,-e_1\}\subset\Sigma_{q}$, and $\{-e_1,e_1\}\subset\Sigma_r$. However, since $p$ has no negative even weight, the first case is impossible. If the second case holds, we must have that $-a=-2e_1$. Moreover, we must have that $a-d=-e_1$ or $c-d=-e_1$. If $a-d=-e_1$, $2e_1-d=a-d=-e_1$ hence $3e_1=d$, which is a contradiction since no additional multiples of $e_1$ should appear. Next, if $c-d=-e_1$, $\{d-c,-e_1,e_1\}=\{-e_1,-e_1,e_1\}\subset\Sigma_{r}$, which is also a contradiction. Hence $e_{i}=1$, for all $i$. Then the weights are
\begin{center}
$\Sigma_{p}=\{-d, a, c, d-a, b-d, 1\} \cup \{-1,1\}^{t}$
$\Sigma_{q}=\{d, -a, b, a-d, -1, 1\} \cup \{-1,1\}^{t}$
$\Sigma_{r}=\{-b,-c,d-b,-1,1,-1\} \cup \{-1,1\}^{t}.$
\end{center}

Second, suppose that $c-d=-1$. Then $N_p(1)=N_p(-1)+1$ and $N_q(1)=N_q(-1)$. Hence by Lemma \ref{l24} for 1, $N_r(1)+1=N_r(-1)$. Considering Lemma \ref{l24} for each integer, one can show that the weights are
\begin{center}
$\Sigma_{p}=\{-d, a, c, d-a, b-d, 1\} \cup \{-1, 1\}_{i=1}^{t}$
$\Sigma_{q}=\{d, a-d, -1, -a, b, 1\} \cup \{-1, 1\}_{i=1}^{t}$
$\Sigma_{r}=\{-b,-c,d-b,-1\} \cup \{-e_{i},e_{i}\}_{i=1}^{t+1}$
\end{center}
for some odd natural numbers $e_{i}$'s. As above, $e_{i}=1$ for all $i$.

Hence in either case the weights are
\begin{center}
$\Sigma_{p}=\{-d, a, c, d-a, b-d, 1\} \cup \{-1,1\}^{t}$
$\Sigma_{q}=\{d, -a, b, a-d, c-d, 1\} \cup \{-1,1\}^{t}$
$\Sigma_{r}=\{-b,-c,d-b,-1,d-c,-1\} \cup \{-1,1\}^{t}.$
\end{center}
Moreover, since $c_{1}(M)|_{p}=0$ by Corollary \ref{c27}, we have that $-d+a+c+d-a+b-d+1=0$. Therefore, $d=c+b+1$. Let $A=c_{n}(M)|_{p}=\prod \xi_{p}^{j}, B=c_{n}(M)|_{q}=\prod \xi_{q}^{j}$, and $C=c_{n}(M)|_{r}=\prod \xi_{r}^{j}$. Then
\begin{center}
$(-1)^{t+1}(B+C)=dab(d-a)(d-c)-bc(d-b)(d-c)$
$=b(d-c)\{da(d-a)-c(d-b)\}$
$=b(d-c)\{(c+b+1)a(c+b+1-a)-c(c+b+1-b)\}$
$=b(d-c)\{(c+b+1)a(2b+1)-c(c+1)\}$
$=b(d-c)\{(a+2b+1)a(2b+1)-(a+b)(a+b+1)\}$
$=b(d-c)\{(a^{2}+2ab+a)(2b+1)-(a^{2}+2ab+b^{2}+a+b)\}$
$=b(d-c)\{2a^{2}b+4ab^{2}+2ab+a^{2}+2ab+a-(a^{2}+2ab+b^{2}+a+b)\}$
$=b(d-c)\{(2a^{2}b-a^{2})+(4ab^{2}-b^{2})+(2ab-2ab)+(a^{2}-a)+(2ab-b)+a\}>0.$
\end{center}
Hence $(-1)^{t+1}B>-(-1)^{t+1}C>0$, i.e., $\D{(-1)^{t+1}\bigg( \frac{1}{B} + \frac{1}{C}\bigg)<0}$. We also have that $\D{(-1)^{t+1}\frac{1}{A}<0}$. Then, by Theorem \ref{t21},
\begin{center}
$\D{0=(-1)^{t+1}\int_M 1 =(-1)^{t+1}\bigg( \frac{1}{A} + \frac{1}{B} + \frac{1}{C}\bigg)<0},$
\end{center}
which is a contradiction.
\end{proof}

\begin{lem} \label{l57} The second case in Lemma \ref{l52} is not possible.
\end{lem}

\begin{proof}By Lemma \ref{l55}, the weights in this case are
$$\Sigma_{p}=\{-d, -a, b, d-b, d-c, 1\} \cup \{-1,1\}_{i=1}^{t}$$
$$\Sigma_{q}=\{d, -b, -c, d-a, b-d, 1\} \cup \{-1,1\}_{i=1}^{t}$$
$$\Sigma_{r}=\{a, c, \cdots \},$$
where the largest weight $d$ is odd, $a,b$, and $c$ are even natural numbers such that $c=a+b$, and $\dim M=12+4t$ for some $t\geq 0$.

Let $A=c_{n}(M)|_{p}=\prod \xi_{p}^{j}, B=c_{n}(M)|_{q}=\prod \xi_{q}^{j}$, and $C=c_{n}(M)|_{r}=\prod \xi_{r}^{j}$. Then
\begin{center}
$(-1)^{t+1}(B+A)=dbc(d-a)(d-b)-dab(d-b)(d-c)$
$=db(d-b)\{c(d-a)-a(d-c)\}>0,$
\end{center}
since $c>a$ and $d-a>d-c.$ Hence it follows that $\D{(-1)^{t+1}\bigg(\frac{1}{A}+\frac{1}{B} \bigg)<0}.$ Since $\lambda_r=\frac{1}{2}\dim M +2$, we also have that $\D{(-1)^{t+1}\frac{1}{C}<0}$. Then, by Theorem \ref{t21},
\begin{center}
$\D{0=(-1)^{t+1}\int_M 1 = (-1)^{t+1}\bigg(\frac{1}{A} + \frac{1}{B} + \frac{1}{C}\bigg)<0},$
\end{center}
which is a contradiction.
\end{proof}

\section{Preliminaries for the largest weight even case part 1}

Let the circle act symplectically on a compact, connected symplectic manifold $M$ with exactly three fixed points. Also assume that $\dim M\geq8$ and the largest weight is even. The main idea to prove Theorem \ref{t11} is to rule out manfiolds such that $\dim M\geq8$. In this section, we investigate properties that the manifold $M$ should satisfy, if it exists.

\begin{pro} \label{p61}
Fix a natural number $n$ such that $n \geq 4$. Assume that Theorem \ref{t11} holds for all manifolds $M$ such that $\dim M < 2n$. Let the circle act symplectically on a $2n$-dimensional compact, connected symplectic manifold $M$ and suppose that there are exactly three fixed points. Assume that the action is effective. Then the largest weight is odd.
\end{pro}

\begin{proof}
Assume on the contrary that the largest weight is even. Then this is an immediate consequence of Lemma \ref{l62}, Lemma \ref{l71}, Lemma \ref{l81}, Lemma \ref{l82}, Lemma \ref{l83}, Lemma \ref{l84}, Lemma \ref{l85}, Lemma \ref{l86}, Lemma \ref{l87}, Lemma \ref{l88}, Lemma \ref{l89}, Lemma \ref{l810}, and Lemma \ref{l811} below.
\end{proof}

\begin{lem} \label{l62}
Fix a natural number $n$ such that $n \geq 4$. Assume that Theorem \ref{t11} holds for all manifolds $M$ such that $\dim M < 2n$. Let the circle act symplectically on a $2n$-dimensional compact, connected symplectic manifold $M$ and suppose that there are exactly three fixed points $p,q$, and $r$, with $\lambda_p \leq \lambda_q \leq \lambda_r$. Assume that the action is effective and the largest weight $c$ is even. Then after possibly reversing the circle action we may assume that the weights are
$$\Sigma_{p}=\{-c,-b\} \cup \{x_{i}\}_{i=1}^{t+3} \cup \{-y_{i}\}_{i=1}^{t} \cup \{1\} \cup \{-1,1\}_{i=1}^{s}$$
$$\Sigma_{q}=\{c,a\} \cup \{x_{i}-c\}_{i=1}^{t+3}\cup\{c-y_{i}\}_{i=1}^{t} \cup \{1\} \cup \{-1,1\}_{i=1}^{s}$$
$$\Sigma_{r}=\{-a,b,\cdots\}$$
for some $s \geq 0$ and $t \geq 0$ such that $\dim M=2n=12+4t+4s$, where $a$ and $b$ are even natural numbers such that $c=a+b$, and $x_i$'s and $y_i$'s are odd natural numbers for all $i$. Moreover, the remaining weights at $r$ are odd.
\end{lem}

\begin{proof}
Let $c$ be the largest weight. By Lemma \ref{l32}, $N_{p}(c)+N_{q}(c)+N_{r}(c)=1$ and $N_{p}(-c)+N_{q}(-c)+N_{r}(-c)=1$. By Lemma \ref{l33}, $\lambda_{p}=n-2, \lambda_{q}=n$, and $\lambda_{r}=n+2$. 
Moreover, after possibly reversing the circle action, we may assume that $-c \in \Sigma_p$ and $c \in \Sigma_q$.

By Lemma \ref{l44}, there exist even natural numbers $a$ and $b$ such that the weights at the three fixed points in the isotropy submanifold $M^{\mathbb{Z}_{2}}$ are $\{a, d\}, \{-a, b\}$, and $\{-b,-d\}$, where $d=a+b$. In Lemma \ref{l44}, the order is not specified. However, since we can assume without loss of generality that $-c \in \Sigma_{p}$ and $c \in \Sigma_{q}$, we can assume that $d=c$, hence $\{-c,-b\}\subset\Sigma_{p},\{c,a\}\subset\Sigma_{q}$, and $\{-a,b\}\subset\Sigma_{r}$. Moreover, these are the only even weights.

Next, by Lemma \ref{l46} part 1 for $c$, $\Sigma_{p} \equiv \Sigma_{q} \mod c$. As a result, we can find a bijection between weights at $p$ and weights at $q$ that takes each weight $\alpha$ at $p$ to a weight $\beta$ at $q$ such that $\alpha \equiv \beta \mod c$.

First, $-c$ at $p$ has to go to $c$ at $q$ since all the other weights are non-zero and have absolute values less than $c$. Second, $-b$ at $p$ must go to $a$ at $q$. Next, if $l$ is any positive odd weight at $p$, then it either goes to $l$ or $l-c$ at $q$. Suppose that there are $t_{0}$ positive odd weights at $p$ that go to negative odd weights at $q$. Then since $\lambda_p=n-2$ and there is no positive even weight at $p$, there are $\frac{n}{2}+1-t_{0}$ positive odd weights $p$ that go to positive odd weights at $q$. Similarly, if $-k$ is any negative odd weight at $p$, either it has to go to $-k$ or $c-k$ at $q$. Suppose that there are $t_{1}$ negative odd weights at $p$ that go to positive odd weights at $q$. On the other hand, since $\lambda_q=n$ and $q$ has two positive even weights, the number of positive odd weights at $q$ that go to positive odd weights at $p$ is equal to $\frac{n}{2}-2-t_{1}$. Hence $\frac{n}{2}+1-t_{0}=\frac{n}{2}-2-t_{1}$, i.e., $t_{0}=t_{1}+3$. Let $t=t_{1}$ and $s=\frac{n}{2} - t - 3$. By Corollary \ref{c23}, $\frac{1}{2}\dim M$ is even. This implies that the weights are
$$\Sigma_{p}=\{-c,-b\} \cup \{x_{i}\}_{i=1}^{t+3} \cup \{-y_{i}\}_{i=1}^{t} \cup \{e_{i}\}_{i=1}^{s+1} \cup \{-f_{i}\}_{i=1}^{s}$$
$$\Sigma_{q}=\{c,a\} \cup \{x_{i}-c\}_{i=1}^{t+3}\cup\{c-y_{i}\}_{i=1}^{t} \cup \{e_{i}\}_{i=1}^{s+1} \cup \{-f_{i}\}_{i=1}^{s}$$
$$\Sigma_r=\{-a,b,\cdots\}$$
for some odd natural numbers $x_{i}$'s, $y_{i}$'s, $e_{i}$'s, and $f_{i}$'s, where $\dim M=2n=12+4t+4s$, for some $t\geq 0$ and $s \geq0$.

Next, we show that $e_i=f_i=1$ for all $i$.

\begin{enumerate}
\item $e_i=1$ for all $i$.

Assume on the contrary that $e_i>1$ for some $i$. Denote $e=e_i$. Then by Lemma \ref{l46} part 3 for $e$, either $\{2e,e\}\subset\Sigma_p$, $\{-e,e\}\subset\Sigma_q$, and $\{-2e,-e\}\subset\Sigma_r$, or $\{-e,e\}\subset\Sigma_p$, $\{2e,e\}\subset\Sigma_q$, and $\{-2e,-e\}\subset\Sigma_r$. Moreover, no additional multiples of $e$ should appear as weights. Since $\{-2e,-e\}\subset\Sigma_r$ in either case, the only possibility is that $a=2e$. Therefore, the latter is the case. Thus, we have that $\{-e,e\}\subset\Sigma_p$. Therefore, $-e=-y_i$ for some $i$ or $-e=-f_i$ for some $i$. Since no additional multiples of $e$ should appear as weights at $q$, $-f_i \neq -e$ for all $i$. Hence, $-y_i=-e$ for some $i$. Without loss of generality, let $y_1=e$. Moreover, since no additional multiples of $e$ should appear as weights, $b \neq 2e$. In particular, $a=2e \neq b$. Since $2e=a<a+b=c$, $e<\frac{c}{2}$. Thus $c-e > \frac{c}{2}$. 

Next, we show that $e-c \in \Sigma_r$. We have that $c-e=c-y_1 \in \Sigma_q$. By Lemma \ref{l24} for $c-e$, either $e-c \in \Sigma_p$, $e-c \in \Sigma_q$, or $e-c \in \Sigma_r$.

\begin{enumerate}[(a)]
\item $e-c \notin \Sigma_p$.

Suppose that $e-c \in \Sigma_p$. Since $e-c$ is odd, either $e-c=-y_i$ for some $i$ or $e-c=-f_i$ for some $i$. 

First, assume that $e-c=-y_i$ for some $i$. If $e-c=-y_1$, this implies that $e-c=-e$ hence $c=2e$, which is a contradiction. Hence if $e-c=-y_i$ for some $i$, $i \neq 1$. Without loss of generality, let $e-c=-y_2$. Then we have that $\{-y_1,-y_2\}=\{e-c,e-c\}\subset\Sigma_p$. Then by Lemma \ref{l46} part 3 for $e-c$, $2(c-e) \in \Sigma_p$, which is a contradiction since $2(c-e) > c$.

Second, assume that $e-c=-f_i$ for some $i$. Then we have that $e-c=-f_i \in \Sigma_p$ and $e-c=-f_i \in \Sigma_q$. Then by Lemma \ref{l46} part 2 for $c-e$, $2(c-e) \in \Sigma_r$, which is a contradiction since $2(c-e) > c$. Hence $e-c \neq \Sigma_p$.

\item $e-c \notin \Sigma_q$.

Suppose that $e-c \in \Sigma_q$. Then we have that $\{c-e,e-c\}=\{c-y_1,e-c\}\subset\Sigma_q$. Hence by Lemma \ref{l46} part 4 for $c-e$, either $2(c-e) \in \Sigma_p$ or $2(c-e) \in \Sigma_r$. However, $2(c-e) >c$, which is a contradiction.
\end{enumerate}

Therefore $e-c \in \Sigma_r$. Then by Lemma \ref{l46} part 1 for $c-e$, $\Sigma_q \equiv \Sigma_r \mod c-e$. Consider $\{c,e\}\subset\Sigma_q$. We have that $c \notin \Sigma_r$ and $e \notin \Sigma_r$. Also, $e-(c-e)=2e-c=a-a-b=-b$, but $-b \notin \Sigma_r$ since $-b$ is a negative even integer and $-a$ is the only negative even weight in $\Sigma_r$, but $a \neq b$. Since $|e+k(c-e)|>c$ for $k<-2$ or $k>1$, $\Sigma_q \equiv \Sigma_r \mod c-e$ and $\{c,e\}\subset\Sigma_q$ imply that $N_r(e-2(c-e))=N_r(3e-2c)=2$. Then by Lemma \ref{l46} part 3 for $3e-2c$, $2(2c-3e) \in \Sigma_r$, which is a contradiction since $2(2c-3e)=4c-6e=c+3c-6e=c+3a+3b-6e=c+6e+3b-6e=c+3b>c$, where $c$ is the largest weight. Therefore, $e_i=1$ for all $i$.

\item $f_i=1$ for all $i$.

Assume on the contrary that $f_i>1$ for some $i$. Denote $f=f_i$. Then by Lemma \ref{l46} part 3 for $f$, either $\{-2f,-f\}\subset\Sigma_p$, $\{-f,f\}\subset\Sigma_q$, and $\{2f,f\}\subset\Sigma_r$, or $\{-f,f\}\subset\Sigma_p$, $\{-2f,-f\}\subset\Sigma_q$, and $\{2f,f\}\subset\Sigma_r$. Moreover, no additional multiples of $f$ should appear as weights. Since $\{2f,f\}\subset\Sigma_r$ in either case, the only possibility is that $b=2f$. Therefore, the former is the case. Thus, we have that $\{-f,f\}\subset\Sigma_q$. Therefore, $f=c-y_i$ for some $i$. Without loss of generality, let $c-y_1=f$. Moreover, since no additional multiples of $f$ should appear as weights, $a \neq 2f$. In particular, $b=2f \neq a$. Since $2f=b<a+b=c$, $f<\frac{c}{2}$. Thus $c-f > \frac{c}{2}$.

Next, we show that $c-f \in \Sigma_r$. We have that $f-c=-y_1 \in \Sigma_p$. By Lemma \ref{l24} for $c-f$, either $c-f \in \Sigma_p$, $c-f \in \Sigma_q$, or $c-f \in \Sigma_r$.

\begin{enumerate}[(a)]
\item $c-f \notin \Sigma_p$.

Suppose that $c-f \in \Sigma_p$. Then we have that $\{c-f,f-c\}=\{c-f,-y_1\}\subset\Sigma_p$. Then by Lemma \ref{l24} part 4 for $c-f$, either $2(c-f) \in \Sigma_q$ or $2(c-f) \in \Sigma_r$, which is a contradiction since $2(c-f) > c$.

\item $c-f \notin \Sigma_q$.

Suppose that $c-f \in \Sigma_q$. Then $c-f=c-y_i$ for some $i$. If $c-f=c-y_1$, then $c-f=c-y_1=f$ hence $c=2f<c$, which is a contradiction. Next, suppose that $c-f=c-y_i$ for some $i \neq 1$. Then we have that $\{f-c,f-c\}=\{-y_1,-y_i\}\subset\Sigma_p$. Hence, by Lemma \ref{l46} part 3 for $f-c$, $2(c-f) \in \Sigma_p$, which is a contradiction since $2(c-f) > c$.
\end{enumerate}

Therefore $c-f \in \Sigma_r$. We also have that $f-c=-y_1 \in \Sigma_p$. Then by Lemma \ref{l46} part 1 for $c-f$, $\Sigma_p \equiv \Sigma_r \mod c-f$. Consider $\{-c,-f\}\subset\Sigma_p$. We have that $-c \notin \Sigma_r$ and $-f \notin \Sigma_r$. Also, $-f+(c-f)=c-2f=a+b-2f=a+2f-2f=a$, but $a \notin \Sigma_r$ since $a$ is a positive even integer and $b$ is the only positive even weight in $\Sigma_r$, but $a \neq b$. Since $|-f+k(c-f)|>c$ for $k<-1$ or $k>2$, $\Sigma_p \equiv \Sigma_r \mod c-f$ and $\{-c,-f\}\subset\Sigma_p$ imply that $N_r(-f+2(c-f))=N_r(2c-3f)=2$. Then by Lemma \ref{l46} part 3 for $2c-3f$, $-2(2c-3f) \in \Sigma_r$, which is a contradiction since $-2(2c-3f)=-4c+6f=-c-3c+6f=-c-3a-3b+6f=-c-3a-6f+6f=-c-3a<-c$, where $-c$ is the smallest weight. Therefore, $f_i=1$ for all $i$.
\end{enumerate}
\end{proof}

\begin{lem} \label{l63} In Lemma \ref{l62}, $x_{i} \neq c-y_{j}$,  for all $i$ and $j$.
\end{lem}

\begin{proof}
Suppose not. Without loss of generality, assume that $x_{1}=c-y_{1}$. Then either $x_{1}>\frac{c}{2}$, $x_{1}-c<-\frac{c}{2}$, or $x_{1}=y_{1}=\frac{c}{2}$. If $x_{1}>\frac{c}{2}$, $x_{1} \in \Sigma_{p}$ and $x_{1}=c-y_{1} \in \Sigma_{q}$. Hence by Lemma \ref{l46} part 2 for $x_1$, $-2x_1 \in \Sigma_r$, which is a contradiction since $-2x_1<-c$ where $-c$ is the smallest weight. Next, assume that $x_{1}-c<-\frac{c}{2}$. Then $x_1-c=-y_1 \in \Sigma_p$ and $x_1-c \in \Sigma_q$. Again by Lemma \ref{l46} part 2 for $x_1-c$, $2(c-x_1) \in \Sigma_r$, which is a contradiction since $2(c-x_{1})>c$ where $c$ is the largest weight. If $x_{1}=y_{1}=\frac{c}{2}$, $\{-2x_{1},x_{1},-x_{1}\}=\{-c,x_{1},-y_{1}\}\subset\Sigma_{p}$, which is a contradiction by Lemma \ref{l46} part 4 for $x_1$.
\end{proof}

\begin{lem} \label{l64}
In Lemma \ref{l62}, if $x_{i}=c-x_{j}$ for $i \neq j$, then $2x_{i}=2x_{j}=c$. Also, if $x_i=c-x_i$ for some $i$, then $2x_i=2x_j=c$ for some $j \neq i$. Moreover, there could be at most one such pair $(x_i,x_j)$ for $i \neq j$ such that $x_i=c-x_j$.
\end{lem}

\begin{proof} First, suppose that $x_i=c-x_i$ for some $i$. Then $c=2x_1$. Thus, we have that $\{-2x_i,x_i\}=\{-c,x_i\}\subset\Sigma_p$ and $\{2x_i,-x_i\}=\{c,x_i-c\}\subset\Sigma_q$. By looking at the isotropy submanifold $M^{\mathbb{Z}_{x_i}}$, this must be the fourth case of Lemma \ref{l45}. Hence, $\{-2x_i,x_i,x_i\}\subset\Sigma_p$ and $\{2x_i,-x_i,-x_i\}\subset\Sigma_q$. This implies that $x_i=x_j$ for some $j \neq i$.

Next, suppose that $x_{i}=c-x_{j}$ and $x_i \neq x_j$ for some $i\neq j$. Without loss of generality, let $x_{1}=c-x_{2}$ and $x_{1}\neq x_{2}$. We can also assume that $x_{1}>x_{2}$. Then $x_{1}>\frac{c}{2}>x_{2}$. Since $x_{1}\in\Sigma_{p}$ and $-x_{1}=x_2-c\in\Sigma_{q}$, by Lemma \ref{l46} part 1 for $x_1$, $\Sigma_{p} \equiv \Sigma_{q} \mod x_{1}$.

First, we can choose a bijection between $\Sigma_p$ and $\Sigma_q$ so that
\begin{center}
$\Sigma_{p}\supset \{1\} \cup \{-1,1\}^{s} \equiv \{1\} \cup \{-1,1\}^{s} \subset \Sigma_{q} \mod x_{1}.$
\end{center}
Also, since $x_{1}+x_{2}=c$, we can also choose so that
\begin{center}
$\Sigma_{p}\supset\ \{-c,x_{1},x_{2}\}=\{-x_1-x_2,x_1,x_2\}$
$\equiv \{-x_{2},-x_{1},x_1+x_2\}=\{x_{1}-c,x_{2}-c,c\}\subset\Sigma_{q} \mod x_1.$
\end{center}
We separate into two cases:
\begin{enumerate}
\item $t=0$.

In this case, we are left with
\begin{center}
$\Sigma_{p}\supset\{-b,x_{3}\}\equiv\{a,x_3-c=x_{3}-x_{1}-x_{2}\}\subset\Sigma_{q} \mod x_{1}.$
\end{center}

If $x_{3}\equiv x_{3}-x_{1}-x_{2} \mod x_{1}$, we have that $x_{1}|x_{2}$, which is a contradiction since $x_{1}>x_{2}$. Hence $x_{3} \equiv a \mod x_{1}$. By Corollary \ref{c27}, $c_{1}(M)|_{p}=-c-b+x_1+x_2+x_3+1=-b+x_3+1=0$, hence $x_{3}+1=b$. Then, since $a=c-b=x_{1}+x_{2}-x_{3}-1$, we have that $x_{3} \equiv x_{1}+x_{2}-x_{3}-1=a \mod x_{1}$, hence $2x_{3}+1 \equiv x_{2} \mod x_{1}$. Since $2x_3+1$, $x_2$, and $x_1$ are odd, and $2x_1>c$ where $c$ is the largest weight, this implies that $2x_{3}+1=x_{2}$. Then we have that $a>x_{1}>\frac{c}{2}>x_{2}>b>x_{3}$. Therefore,
\begin{center}
$\D{\frac{-(-1)^{s}(B+A)}{x_{1}x_{2}(x_{1}+x_{2})}=a(x_{1}+x_{2}-x_{3})-bx_{3} > 0},$
\end{center}
hence $(-1)^{s+1}B>(-1)^{s}A$. Also, since $\lambda_r=\frac{1}{2}\dim M +2$ by Lemma \ref{l33}, $(-1)^{s}C>0$. Then, by Theorem \ref{t21},
\begin{center}
$\D{0=\int_M 1 = (-1)^{s} \bigg( \frac{1}{A} + \frac{1}{B} + \frac{1}{C} \bigg) >0},$
\end{center}
which is a contradiction.

\item $t>0$.

In this case, we are left with
\begin{center}
$\{-b \} \cup \{x_{i}\}_{i=3}^{t+3} \cup \{-y_{i}\}_{i=1}^{t}$
$\equiv \{a\} \cup \{x_{i}-c\}_{i=3}^{t+3} \cup \{c-y_i\}_{i=1}^{t}\mod x_{1}.$
\end{center}
Without loss of generality, let $x_{3}\leq x_4 \leq \cdots\leq x_{t+3}$ and $-y_{1}\leq -y_2 \leq \cdots \leq-y_{t}$. Hence,
\begin{center}
$\{-b, -y_{1}\leq \cdots\leq -y_{t}<0<x_{3}\leq \cdots\leq x_{t+3}\}$
$\equiv \{x_{3}-c\leq\cdots\leq x_{t+3}-c<0<c-y_{1}\leq\cdots\leq c-y_{t}, a\} \mod x_{1}$
\end{center}
Recall that $x_1$ is odd and $2x_1 >c$ where $c$ is the largest weight.

Consider $x_{3}\in \Sigma_p$. If $x_3 \equiv c-y_i \mod x_1$ for some $i$, then $x_3=c-y_i$, which contradicts Lemma \ref{l63}. If $x_3 \equiv x_i-c \mod x_1$ for some $i \neq 1$ and 2, we have that $x_3-2x_1=x_i-c$, which is a contradiction since $x_3-2x_1 < x_i-c$ for $i \neq 1$ and 2. Hence, $x_3 \equiv a \mod x_1$. Then we can also choose so that $ -b = a-c \equiv x_3-c  \mod x_1$. Then we are left with
\begin{center}
$\{-y_{1}\leq \cdots\leq -y_{t}<0<x_{4} \leq \cdots \leq x_{t+3}\}$
$\equiv \{x_{4}-c\leq\cdots\leq x_{t+3}-c<0<c-y_{1}\leq\cdots\leq c-y_{t}\} \mod x_{1}$
\end{center}

Next, consider $-y_{t} \in \Sigma_p$. If $-y_t \equiv x_i-c \mod x_1$ for some $i \neq 1,2$, and 3, then $-y_t=x_i-c$, which is a contradiction by Lemma \ref{l63}. If $-y_t \equiv c-y_i \mod x_1$ for some $i$, then $2x_1-y_t=c-y_i$, which is a contradiction since $2x_1-y_t>c-y_i$ for all $i$. Then $-y_t \in \Sigma_r$ is congruent to no element in $\Sigma_q$ modulo $x_1$, which is a contradiction.
\end{enumerate}

Finally, without loss of generality, assume that $x_1=c-x_2$ and $x_3=c-x_i$ for some $i$. Then $2x_1=2x_2=c$ and $2x_3=2x_i=c$ for some $i$. Then we have $\{-2x_{1},x_{1},x_{1},x_{1}\}=\{-c,x_1,x_2,x_3\}\subset\Sigma_{p}$, which is a contradiction by Lemma \ref{l46} part 3 for $x_1$.
\end{proof}

\begin{lem} \label{l65} In Lemma \ref{l62}, $y_{i}\neq c-y_{j}$, if $i\neq j$.
\end{lem}

\begin{proof}
Suppose not. First, assume that $y_i=c-y_j$ and $y_i=y_j$ for some $i \neq j$ , i.e., $2y_i=2y_j=c$. Then $\{-2y_i,-y_i,-y_i\}=\{-c,-y_i,-y_j\}\subset\Sigma_p$, which contradicts Lemma \ref{l46} part 3 for $y_i$.

Second, assume that $y_i=c-y_j$ and $y_i \neq y_j$ for some $i \neq j$. Without loss of generality assume that $y_{1}= c-y_{2}$ and $y_{1}\neq y_{2}$. We can also assume that $y_{1} > \frac{c}{2} > y_{2}$. Then we have that $-y_1 \in \Sigma_p$ and $c-y_2=y_1 \in \Sigma_q$. Hence by Lemma \ref{l46} part 1 for $y_1$, $\Sigma_{p} \equiv \Sigma_{q} \mod y_{1}$.

First, we can choose a bijection between $\Sigma_p$ and $\Sigma_q$ so that
\begin{center}
$\Sigma_{p}\supset\{-y_1,1\} \cup \{-1,1\}^{s} \equiv\{c-y_2=y_1,1\} \cup \{-1,1\}^{s} \subset \Sigma_{q} \mod y_{1}.$
\end{center}
Then we are left with
\begin{center}
$\Sigma_{p}\supset\{-c,-b \} \cup \{x_{i}\}_{i=1}^{t+3} \cup \{-y_{i}\}_{i=2}^{t}$
$ \equiv \{c,a\} \cup \{x_{i}-c\}_{i=1}^{t+3} \cup\{c-y_1=y_2\}\cup\{c-y_i\}_{i=3}^{t} \subset \Sigma_{q}\mod y_{1}.$
\end{center}
Without loss of generality, let $x_{1}\leq x_2 \leq\cdots\leq x_{t+3}$ and $-y_{3}\leq -y_4 \leq \cdots \leq-y_{t}$, i.e.,
\begin{center}
$\{-c,-b, -y_{2}, -y_{3}\leq -y_4 \leq \cdots\leq -y_{t}<0<x_{1}\leq x_2 \leq \cdots \leq x_{t+3}\}$
$\equiv \{x_{1}-c\leq \cdots \leq x_{t+3}-c<0<c-y_{3}\leq\cdots\leq c-y_{t}, c,a,c-y_{2}\} \mod x_{1}.$
\end{center}
Recall that $y_1$ is odd and $2y_1>c$ where $c$ is the largest weight.

Consider $x_1 \in \Sigma_p$. If $x_1 \equiv c-y_i \mod y_1$ for some $i\neq 1$ and 2, then $x_1=c-y_i$, which contradicts Lemma \ref{l63}. If $x_1 \equiv x_i-c$ for some $i$, $x_1-2y_1 =x_i-c$, which is a contradiction since $x_1-2y_1 < x_i-c$ for all $i$. If $x_1 \equiv c \mod y_1$, then $x_1+y_1=c$, which contradicts Lemma \ref{l63}. Therefore, $x_1 \equiv a \mod y_1$.

Next, consider $-y_t \in \Sigma_p$. If $-y_t \equiv c-y_i$ for some $i \neq 1$ and 2, then $2y_1-y_t=c-y_i$, which is a contradiction since $2y_1-y_t > c-y_i$ for all $i$. If $-y_t \equiv x_i-c \mod y_1$ for some $i$, then $-y_t=x_i-c$, which contradicts Lemma \ref{l63}. Therefore, we have that either $-y_t \equiv c \mod y_1$ or $-y_t \equiv c-y_2 \mod y_1$.

Suppose that $-y_{t} \equiv c \mod y_{1}$. This means that $-y_{t}+3y_{1}=c$. Then we have that $-y_t=c-3y_1=y_1+y_2-3y_1=y_2-2y_1<y_1-2y_1=-y_1<-y_2$, hence $-y_t<-y_2$. Next, we consider $-y_2 \in \Sigma_p$. Using the same argument for $-y_t$, we have that $-y_2 \in \Sigma_p$ is congruent to no element in $\Sigma_q$ modulo $y_1$, which is a contradiction.

Next, suppose that $-y_t \equiv c-y_2 \mod y_1$. This means that $2y_1-y_t=c-y_2$. Then we have that $2y_1-y_t=c-y_2=y_1+y_2-y_2=y_1$, hence $y_1=y_t$. Hence, we have $\{-y_1,-y_1\}=\{-y_1,-y_t\}\subset\Sigma_p$. Then by Lemma \ref{l46} part 3 for $-y_1$, $2y_{1} \in \Sigma_{p}$, which is a contradiction since $c < 2y_{1}$ where $c$ is the largest weight.
\end{proof}

\begin{lem} \label{l66} Fix a natural number $e$ such that $e \neq \frac{c}{2}$. In Lemma \ref{l62}, at most one of $x_i$'s, $y_i$'s, $c-x_i$'s, and $c-y_i$'s can be $e$.
\end{lem}

\begin{proof}
Fix a natural number $e$ such that $e \neq \frac{c}{2}$.

\begin{enumerate}
\item $x_i=e$ for some $i$.

First, by Lemma \ref{l63}, $x_i \neq c-y_j$ for all $j$. Second, suppose that $x_i=c-x_j$ for some $j$. Then by Lemma \ref{l64}, $2e=2x_i=2x_j=c$, which is a contradiction.

Third, suppose that $x_i=y_j$ for some $j$. Assume that $e > \frac{c}{2}$. Since $\{-e,e\}=\{-y_j,x_i\}\subset\Sigma_p$, either $2e \in \Sigma_q$ or $2e \in \Sigma_r$ by Lemma \ref{l46} part 4 for $e$, which is a contradiction since $2e>c$ where $c$ is the largest weight. Next, assume that $e<\frac{c}{2}$. Since $\{e-c,c-e\}=\{x_i-c,c-y_j\}\subset\Sigma_q$, either $2(c-e) \in \Sigma_p$ or $2(c-e) \in \Sigma_r$ by Lemma \ref{l46} part 4 for $c-e$, which is a contradiction since $2(c-e)>c$ where $c$ is the largest weight. 

Last, suppose that $x_i=x_j$ for some $j \neq i$. Assume that $e > \frac{c}{2}$. Since $\{e,e\}=\{x_i,x_j\}\subset\Sigma_p$, $-2e \in \Sigma_p$ by Lemma \ref{l46} part 3 for $e$, which is a contradiction since $-2e<-c$ where $-c$ is the smallest weight. Next, assume that $e<\frac{c}{2}$. Since $\{e-c,e-c\}=\{x_i-c,x_j-c\}\subset\Sigma_q$, $2(c-e) \in \Sigma_q$ by Lemma \ref{l46} part 3 for $e-c$, which is a contradiction since $2(c-e)>c$ where $c$ is the largest weight.

\item $y_i=e$ for some $i$.

As above, $y_i \neq x_j$ for all $j$. By Lemma \ref{l63}, $y_i \neq c-x_j$ for all $j$.

Next, suppose that $y_i=c-y_j$ for some $j$. Then by Lemma \ref{l64}, $i=j$. Hence $c=2y_i=2e$, which is a contradiction by the assumption that $e \neq \frac{c}{2}$.

Finally, Suppose that $y_i=y_j$ for some $j \neq i$. Assume that $e > \frac{c}{2}$. Since $\{-e,-e\}=\{-y_i,-y_j\}\subset\Sigma_p$, $2e \in \Sigma_p$ by Lemma \ref{l46} part 3 for $e$, which is a contradiction since $2e>c$ where $c$ is the largest weight. Next, assume that $e<\frac{c}{2}$. Since $\{c-e,c-e\}=\{c-y_i,c-y_j\}\subset\Sigma_q$, $-2(c-e) \in \Sigma_q$ by Lemma \ref{l46} part 3 for $c-e$, which is a contradiction since $-2(c-e)<-c$ where $-c$ is the smallest weight.

\item $c-x_i=e$ for some $i$.

As above, $c-x_i \neq x_j$ for all $j$ and $c-x_i \neq y_j$ for all $j$. Since $x_j \neq y_k$ for all $j$ and $k$, $c-x_i \neq c-y_j$ for all $j$. Also, since $x_j \neq x_k$ for all $j$ and $k$, $c-x_i \neq c-x_j$ for all $j$.

\item $c-y_i=e$ for some $i$.

As above, $c-y_i \neq x_j$, $c-y_i \neq y_j$, and $c-y_i \neq c-x_j$ for all $j$. Since $y_j \neq y_k$ for all $j$ and $k$ as above, $c-y_i \neq c-y_j$ for all $j$.
\end{enumerate}
\end{proof}

\begin{lem} \label{l67} In Lemma \ref{l62}, $x_i \neq y_j$ for all $i$ and $j$.
\end{lem}

\begin{proof}
Assume on the contrary that $x_i = y_j$ for some $i$ and $j$. Then by Lemma \ref{l66}, $x_i=y_j=\frac{c}{2}$. Hence, we have $\{-2x_1,x_1,-x_1\}=\{-c,x_1,-y_1\}\subset\Sigma_p$, which contradicts Lemma \ref{l46} part 4 for $x_1$.
\end{proof}

\begin{lem} \label{l68} In Lemma \ref{l62}, assume that $\{-f,f\} \subset \Sigma_r$ for some natural number $f$. If $f>1$, then $c=2f$, $\{-c=-2f,-f\}\subset\Sigma_p$, $\{2f=c,f\}\subset\Sigma_q$, and $\{-f,f\}\subset\Sigma_r$. Moreover, no additional multiples of $f$ should appear as weights.
\end{lem}

\begin{proof}
Assume that $f>1$. By Lemma \ref{l46} part 4 for $f$, either $\{2f,f\}\subset\Sigma_p$, $\{-2f,-f\}\subset\Sigma_q$, and $\{-f,f\}\subset\Sigma_r$, or $\{-2f,-f\}\subset\Sigma_p$, $\{2f,f\}\subset\Sigma_q$, and $\{-f,f\}\subset\Sigma_r$. However, since $\Sigma_p$ does not have a positive even weight, the former case is impossible. Hence the latter must be the case. Then, $-2f \in \Sigma_p$ implies that $c=2f$ or $b=2f$. Suppose that $b=2f$. Then we have that $\{b=2f,-f,f\}\subset\Sigma_r$, which is a contradiction by Lemma \ref{l46} part 4 for $f$. Therefore, $c=2f$.
\end{proof}

\begin{lem} \label{l69} In Lemma \ref{l62}, if $N_{p}(1) > N_{r}(1)$ and $N_{q}(1) >N_{r}(1)$, then $N_{p}(1) < N_{r}(1)+3$ or $N_{q}(1) <N_{r}(1)+3$. Similarly, if $N_{p}(-1) > N_{r}(-1)$ and $N_{q}(-1) >N_{r}(-1)$, then $N_{p}(-1) < N_{r}(-1)+3$ or $N_{q}(-1) <N_{r}(-1)+3$.
\end{lem}

\begin{proof} 
First we prove the former. For this suppose not, i.e., $N_{p}(1) \geq N_{r}(1)+3$ and $N_{q}(1) \geq N_{r}(1)+3$. There are three cases:
\begin{enumerate}
\item $a>\frac{c}{2}$.

Since $a \in \Sigma_q$ and $-a \in \Sigma_r$, by Lemma \ref{l46} part 1 for $a$, $\Sigma_{q} \equiv \Sigma_{r} \mod a$. With $N_{q}(1) \geq N_{r}(1)+3$, this implies that $N_r(1+a)\geq2$ or $N_r(1-a)\geq2$, since $|1+ka|>c$ for $|k|\geq 2$. If $N_r(1+a)\geq2$, $-2(1+a) \in \Sigma_{r}$ by Lemma \ref{l46} part 3 for $1+a$, but $2(1+a)>c$, which is a contradiction.
If $N_r(1-a)\geq2$, $2(a-1) \in \Sigma_{r}$ by Lemma \ref{l46} part 3 for $1-a$. However, $2(a-1) \geq c$ but $c \notin \Sigma_{r}$.

\item $a<\frac{c}{2}$.

Suppose that $a<\frac{c}{2}$. Then $b=c-a>\frac{c}{2}$. Since $-b \in \Sigma_p$ and $b \in \Sigma_r$, by Lemma \ref{l46} part 1 for $b$, $\Sigma_{p} \equiv \Sigma_{r} \mod b$. With $N_{p}(1) \geq N_{r}(1)+3$, this implies that $N_r(1+b)\geq2$ or $N_r(1-b)\geq2$, since $|1+kb|>c$ for $|k|\geq 2$. If $N_r(1+b)\geq2$, $-2(1+b) \in \Sigma_{r}$ by Lemma \ref{l46} part 3 for $1+b$. However, $2(1+b)>c$, which is a contradiction. If $N_r(1-b)\geq2$, $2(b-1) \in \Sigma_{r}$ by Lemma \ref{l46} part 3 for $1-b$. However, $2(b-1) \geq c$ but $c \notin \Sigma_{r}$.

\item $a=\frac{c}{2}$.

Since $a=\frac{c}{2}$, we have that $b=c-a=\frac{c}{2}$. Then the isotropy submanifold $M^{\mathbb{Z}_a}$ must be the third case of Lemma \ref{l45}. This means that the three fixed point lie in the same component of $M^{\mathbb{Z}_a}$, hence $\Sigma_p \equiv \Sigma_q \equiv \Sigma_r \mod a$ by Lemma \ref{l25}. With $N_{q}(1) \geq N_{r}(1)+3$, $\Sigma_q \equiv \Sigma_r \mod a$ implies that $N_r(1+a)\geq2$, $N_r(1-a)\geq2$, $N_r(1-2a)\geq2$, or $N_r(1+a)=N_r(1-a)=N_r(1-2a)=1$, since $|1+ka|>c$ for $k \neq -2,-1,0$, and $1$.

\begin{enumerate}[(a)]

\item $N_r(1+a)\geq2$.

Since $N_r(1+a)\geq2$, $-2(1+a) \in \Sigma_{r}$ by Lemma \ref{l46} part 3 for $1+a$. However, $-2(1+a)<-c$ where $-c$ is the smallest weight, which is a contradiction.

\item $N_r(1-2a)\geq2$.

By Lemma \ref{l46} part 3 for $1-2a$, $2(2a-1) \in \Sigma_r$. However, $2(2a-1)=2(c-1)>c$ where $c$ is the largest weight, which is a contradiction.

\item $N_r(1-a)\geq2$.

By Lemma \ref{l46} part 3 for $1-a$, $N_r(1-a)=2$ and $2(a-1) \in \Sigma_{r}$. Since $b$ is the only positive even weight at $r$, this means that $2(a-1)=b$. Hence $a=b=2$ and $c=a+b=4$. Then the weights at $p$ and $q$ are
\begin{center}
$\Sigma_{p}=\{-4,-2\} \cup \{x_{i}\}_{i=1}^{t+3} \cup \{-y_{i}\}_{i=1}^{t} \cup \{1\} \cup \{-1,1\}^{s}$
$\Sigma_{q}=\{4,2\} \cup \{x_{i}-4\}_{i=1}^{t+3}\cup\{4-y_{i}\}_{i=1}^{t} \cup \{1\} \cup \{-1,1\}^{s}.$
\end{center}
Since $c=4$ is the largest weight, all of $x_{i}$'s and $y_{i}$'s are either $1$ or $3$. If at least two of $x_{i}$'s are 3, $N_p(3)\geq 2$ hence $-6 \in \Sigma_p$ by Lemma \ref{l46} part 3 for 3, which is a contradiction since $-4$ is the smallest weight. If at most one of $x_{i}$ is 1, then at least two $x_i-c$'s are -3. This means that $N_q(-3) \geq 2$ and hence $6 \in \Sigma_{q}$ by Lemma \ref{l46} part 3 for $-3$, which is a contradiction since 4 is the largest weight.

\item $N_r(1+a)=N_r(1-a)=N_r(1-2a)=1$.

Since $1-2a \in \Sigma_r$, by Lemma \ref{l24} for $1-2a$, there must be a weight of $2a-1$ for some fixed point. If it is $r$, $\{2a-1,1-2a\}\subset\Sigma_r$. Hence $2(2a-1) \in \Sigma_p$ or $2(2a-1) \in \Sigma_q$ by Lemma \ref{l46} part 4 for $2a-1$, which is a contradiction since $2(2a-1)=2(c-1) > c$ where $c$ is the largest weight. Hence either $2a-1 \in \Sigma_p$ or $2a-1 \in \Sigma_q$. Suppose that $2a-1 \in \Sigma_p$. Then by Lemma \ref{l46} part 1 for $2a-1$, $\Sigma_p \equiv \Sigma_r \mod 2a-1$. That $N_{p}(1) \geq N_{r}(1)+3$ implies that $N_r(2-2a) \geq 3$ since $|1+k(2a-1)|\geq c$ for $k \neq 0$ and $-1$, and the fixed point $r$ does not have a weights of $c$. However, $r$ has only one negative even weight $-a$, which is a contradiction. Similarly, $2a-1 \in \Sigma_q$ is also impossible.
\end{enumerate}
\end{enumerate}
With a slight variation of this argument, one can prove the latter.
\end{proof}

\section{Preliminaries for the largest weight even case part 2}

Let the circle act symplectically on a compact, connected symplectic manifold $M$ with exactly three fixed points. Also, assume that $\dim M\geq8$ and the largest weight is even. In this section, for technical reasons, we consider $w=\min_{\alpha \in M^{S^1}} \min \{ N_{\alpha}(-1), N_{\alpha}(1) \}$ and rewrite the weights in terms of $w$. And then we further investigate properties that the manifold $M$ should satisfy in terms of $w$, if such a manifold exists.

\begin{lem} \label{l71} Let $w=\min_{\alpha \in M^{S^1}} \min \{ N_{\alpha}(-1), N_{\alpha}(1) \}$. In Lemma \ref{l62}, the weights are
\begin{center}
$\Sigma_{p}=\{-c,-b\} \cup \{x_{i}\}_{i=1}^{t+3} \cup \{-y_{i}\}_{i=1}^{t} \cup \{1\}\cup\{-1,1\}^{v+w}$
$\Sigma_{q}=\{c,a\} \cup \{x_{i}-c\}_{i=1}^{t+3}\cup\{c-y_i\}_{i=1}^{t} \cup \{1\}\cup\{-1,1\}^{v+w}$
$\Sigma_{r}=\{-a,b,\cdots\}\cup\{-1,1\}^{w}$
\end{center}
for some $t\geq0$ and $v\geq0$, where $a,b,$ and $c$ are even natural numbers such that $c=a+b$ is the largest weight, and $x_i$'s and $y_i$'s are odd natural numbers for all $i$. Moreover, the remaining weights at $r$ are odd.
\end{lem}

\begin{proof} In Lemma \ref{l62}, the weights are
$$\Sigma_{p}=\{-c,-b\} \cup \{x_{i}\}_{i=1}^{t+3} \cup \{-y_{i}\}_{i=1}^{t}  \cup \{1\} \cup \{-1,1\}^{s}$$
$$\Sigma_{q}=\{c,a\} \cup \{x_{i}-c\}_{i=1}^{t+3}\cup\{c-y_{i}\}_{i=1}^{t} \cup \{1\} \cup \{-1,1\}^{s}$$
$$\Sigma_{r}=\{-a,b,\cdots\}$$
for some $s \geq 0$ and $t \geq 0$ such that $\dim M=2n=12+4t+4s$, where $a,b,$ and $c$ are even natural numbers such that $c=a+b$ is the largest weight, and $x_i$'s and $y_i$'s are odd natural numbers for all $i$. Moreover, the remaining weights at $r$ are odd.

Let $w=\min_{\alpha \in M^{S^1}} \min \{ N_{\alpha}(-1), N_{\alpha}(1) \}$. We rewrite the weights in terms of $w$. We show that $\{-c,-b\} \cup \{x_{i}\}_{i=1}^{t+3} \cup \{-y_{i}\}_{i=1}^{t}  \cup \{1\}$ in $\Sigma_p$ and $\{c,a\} \cup \{x_{i}-c\}_{i=1}^{t+3}\cup\{c-y_i\}_{i=1}^{t} \cup \{1\}$ in $\Sigma_q$ do not contribute to $w$, i.e.,
\begin{center}
$\{-1,1\} \nsubseteq (\{-c,-b\} \cup \{x_{i}\}_{i=1}^{t+3} \cup \{-y_{i}\}_{i=1}^{t}  \cup \{1\}) \cap (\{c,a\} \cup \{x_{i}-c\}_{i=1}^{t+3}\cup\{c-y_i\}_{i=1}^{t} \cup \{1\}).$
\end{center}
First, $a$, $b$, and $c$ are even natural numbers. Second, by Lemma \ref{l66}, at most one of $c-x_i$'s or $y_i$'s can be 1.

Suppose that $y_i=1$ for some $i$. Then $c-x_j \neq 1$ for all $j$ by Lemma \ref{l66}. Hence in $\Sigma_q$, $\{-1,1\} \nsubseteq (\{c,a\} \cup \{x_{i}-c\}_{i=1}^{t+3}\cup\{c-y_i\}_{i=1}^{t} \cup \{1\})$.

Next, suppose that $c-x_i=1$ for some $i$. Then $y_j \neq 1$ for all $j$ by Lemma \ref{l66}. Hence in $\Sigma_p$, $\{-1,1\} \nsubseteq (\{-c,-b\} \cup \{x_{i}\}_{i=1}^{t+3} \cup \{-y_{i}\}_{i=1}^{t}  \cup \{1\})$.

Last, if $y_i \neq 1$ and $c-x_i \neq 1$ for all $i$, then $\{-1,1\} \nsubseteq (\{-c,-b\} \cup \{x_{i}\}_{i=1}^{t+3} \cup \{-y_{i}\}_{i=1}^{t}  \cup \{1\})$ in $\Sigma_p$ and $\{-1,1\} \nsubseteq (\{c,a\} \cup \{x_{i}-c\}_{i=1}^{t+3}\cup\{c-y_i\}_{i=1}^{t} \cup \{1\})$ in $\Sigma_q$.

Therefore, we can rewrite the weights so that the weights are
\begin{center}
$\Sigma_{p}=\{-c,-b\} \cup \{x_{i}\}_{i=1}^{t+3} \cup \{-y_{i}\}_{i=1}^{t} \cup \{1\}\cup\{-1,1\}^{v+w}$
$\Sigma_{q}=\{c,a\} \cup \{x_{i}-c\}_{i=1}^{t+3}\cup\{c-y_i\}_{i=1}^{t} \cup \{1\}\cup\{-1,1\}^{v+w}$
$\Sigma_{r}=\{-a,b,\cdots\}\cup\{-1,1\}^{w}$
\end{center}
where $s=v+w$.
\end{proof}

\begin{lem} \label{l72} In Lemma \ref{l71}, for each $x_i$, either $x_i=c-x_j$ for some $j$ or $-x_i \in \Sigma_r\setminus(\{-a,b\}\cup\{-1,1\}^w)$.
\end{lem}

\begin{proof}
By Lemma \ref{l24}, for each $x_i$, either $-x_i \in \Sigma_p$, $-x_i \in \Sigma_q$, or $-x_i \in \Sigma_r$. 

First, assume that $x_i >1$. By Lemma \ref{l66}, $x_i \neq y_j$ for all $j$. Hence $-x_i \notin \Sigma_p$. Next, if $-x_i \in \Sigma_q$, then $-x_i=x_j-c$ for some $j$. If $-x_i \in \Sigma_r$, $-x_i \in \Sigma_r\setminus(\{-a,b\}\cup\{-1,1\}^w)$.

Second, assume that $x_i=1$. By Lemma \ref{l66}, at most one of $c-x_i$'s or $y_i$'s can be $1$. Hence, either $N_p(1)>N_p(-1)$ and $N_q(1)\geq N_q(-1)$, or $N_p(1)\geq N_p(-1)$ and $N_q(1)> N_q(-1)$. By Lemma \ref{l24}, this implies that $N_r(1)<N_r(-1)$. Therefore, $-x_i=-1 \in \Sigma_r\setminus(\{-a,b\}\cup\{-1,1\}^w)$.
\end{proof}

\begin{lem} \label{l73} In Lemma \ref{l71}, for each $c-y_i$, $y_i-c \in \Sigma_r\setminus(\{-a,b\}\cup\{-1,1\}^w)$.
\end{lem}

\begin{proof}
By Lemma \ref{l24}, for each $c-y_i$, either $y_i-c \in \Sigma_p$, $y_i-c \in \Sigma_q$, or $y_i-c \in \Sigma_r$.

First, assume that $c-y_i>1$. Suppose that $c-y_i \in \Sigma_q$. Then $c-y_i = c-x_j$ for some $j$, which is a contradiction since $c-y_i \neq c-x_j$ for all $j$ by Lemma \ref{l67}. Next, suppose that $y_i-c \in \Sigma_p$. Then $y_i-c=-y_j$ for some $j$. By Lemma \ref{l65}, $i=j$, i.e., $c=2y_i$. Hence, $\{-2y_i,-y_i\}=\{-c,-y_i\}\subset\Sigma_p$ and $\{2y_i,y_i\}=\{c,c-y_i\}\subset\Sigma_q$. The isotropy submanifold $M^{\mathbb{Z}_{y_i}}$ must be the third case of Lemma \ref{l45}. Therefore, we have that $\{-y_i,y_i\}\subset\Sigma_r$. In particular, $y_i-c=-y_i \in \Sigma_r\setminus(\{-a,b\}\cup\{-1,1\}^w)$, since $c-y_i=y_i=\frac{c}{2}\geq 2$ and $c-y_i$ is odd.

Second, assume that $c-y_i=1$. By Lemma \ref{l66}, at most one of $c-x_i$'s or $y_i$'s can be $1$. Hence, either $N_p(1)>N_p(-1)$ and $N_q(1)\geq N_q(-1)$, or $N_p(1)\geq N_p(-1)$ and $N_q(1)> N_q(-1)$. By Lemma \ref{l24}, this implies that $N_r(1)<N_r(-1)$. Therefore, $y_i-c=-1 \in \Sigma_r\setminus(\{-a,b\}\cup\{-1,1\}^w)$.
\end{proof}

\begin{lem} \label{l74} In Lemma \ref{l71}, suppose that $c-x_i \neq 1$. Then either $c-x_i=x_j$ for some $j$ or $c-x_i \in \Sigma_r\setminus(\{-a,b\}\cup\{-1,1\}^w)$.
\end{lem}

\begin{proof}
Suppose that $c-x_i \neq 1$. By Lemma \ref{l24}, for each $c-x_i$, either $c-x_i \in \Sigma_p$, $c-x_i \in \Sigma_q$, or $c-x_i \in \Sigma_r$. First, assume that $c-x_i \in \Sigma_p$. Then $c-x_i=x_j$ for some $j$. Second, assume that $c-x_i \in \Sigma_q$. Then $c-x_i=c-y_j$ for some $j$, which is a contradiction by Lemma \ref{l67}. Hence $c-x_i \notin \Sigma_q$. Last, assume that $c-x_i \in \Sigma_r$. Since $c-x_i \neq 1$, $c-x_i \in \Sigma_r\setminus(\{-a,b\}\cup\{-1,1\}^w)$.
\end{proof}

\begin{lem} \label{l75} In Lemma \ref{l71}, suppose that $y_i \neq 1$. Then $y_i \in \Sigma_r\setminus(\{-a,b\}\cup\{-1,1\}^w)$.
\end{lem}

\begin{proof}
Suppose that $y_i \neq 1$. By Lemma \ref{l24}, for each $y_i$, either $y_i \in \Sigma_p$, $y_i \in \Sigma_q$, or $y_i \in \Sigma_r$. First, assume that $y_i \in \Sigma_p$. Then $y_i=x_j$ for some $j$, which is a contradiction by Lemma \ref{l67}. Hence $y_i \notin \Sigma_p$. Second, assume that $y_i \in \Sigma_q$. Then $y_i=c-y_j$ for some $j$. By Lemma \ref{l65}, $i=j$, i.e., $c=2y_i$. Hence, $\{-2y_i,-y_i\}=\{-c,-y_i\}\subset\Sigma_p$ and $\{2y_i,y_i\}=\{c,c-y_i\}\subset\Sigma_q$. The isotropy submanifold $M^{\mathbb{Z}_{y_i}}$ must be the third case of Lemma \ref{l45}. Therefore, we have that $\{-y_i,y_i\}\subset\Sigma_r$. In particular, $y_i \in \Sigma_r\setminus(\{-a,b\}\cup\{-1,1\}^w)$, since $y_i=\frac{c}{2} \geq 2$ and $y_i$ is odd. Last, assume that $y_i \in \Sigma_r$. Since $y_i \neq 1$, $y_i \in \Sigma_r\setminus(\{-a,b\}\cup\{-1,1\}^w)$.
\end{proof}

\begin{lem} \label{l76} In Lemma \ref{l71}, if $x_i \neq c-x_j$ for all $i$ and $j$, then $t < v$.
\end{lem}

\begin{proof}
Assume on the contrary that $x_i \neq c-x_j$ for all $i$ and $j$, and $t \geq v$. By Lemma \ref{l72}, $-x_i \in \Sigma_r\setminus(\{-a,b\}\cup\{-1,1\}^w)$ for all $i$. Also, by Lemma \ref{l73}, $y_i-c \in \Sigma_r\setminus(\{-a,b\}\cup\{-1,1\}^w)$ for all $i$.

We show that $x_i \neq x_j$ and $y_i \neq y_j$ for $i \neq j$. Suppose that $x_i=x_j$ for some $i \neq j$. Then by Lemma \ref{l66}, $2x_i=2x_j=c$, hence $x_i=c-x_j$, which contradicts the assumption. Therefore, $x_i \neq x_j$ for $i \neq j$. Suppose that $y_i=y_j$ for some $i \neq j$. Then by Lemma \ref{l66}, $2y_i=2y_j=c$, hence $y_i=c-y_j$, which contradicts Lemma \ref{l65}. Therefore, $y_i \neq y_j$ for $i \neq j$.

First, suppose that $c-x_i=1$ for some $i$. Without loss of generality, let $c-x_1=1$. Then by Lemma \ref{l66}, $c-x_i \neq 1$ for $i \neq 1$ and $y_j \neq 1$ for all $j$. Hence $c-x_i \in \Sigma_r\setminus(\{-a,b\}\cup\{-1,1\}^w)$ for $i \neq 1$ by Lemma \ref{l74}. Also, $y_j \in \Sigma_r\setminus(\{-a,b\}\cup\{-1,1\}^w)$ for all $j$ by Lemma \ref{l75}. Also, by Lemma \ref{l63}, $x_i \neq c-y_j$ for all $i$ and $j$. Therefore, we have that $\{-x_i\}_{i=1}^{t+3} \cup \{c-x_i\}_{i=2}^{t+3} \cup \{y_i\}_{i=1}^t \cup \{c-y_i\}_{i=1}^t \subset \Sigma_r\setminus(\{-a,b\}\cup\{-1,1\}^w)$, which is a contradiction since $|\{-x_i\}_{i=1}^{t+3} \cup \{c-x_i\}_{i=2}^{t+3} \cup \{y_i\}_{i=1}^t \cup \{c-y_i\}_{i=1}^t|=4t+5$ and $|\Sigma_r\setminus(\{-a,b\}\cup\{-1,1\}^w)|=2t+4+2u+2v$, but $4t+5 =2t+2t+5>2t+4+2v$.

Second, suppose that $y_i=1$ for some $i$. Without loss of generality, let $y_1=1$. Then by Lemma \ref{l66}, $c-x_i \neq 1$ for all $i$ and $y_j \neq 1$ for $j\neq1$. Hence $c-x_i \in \Sigma_r\setminus(\{-a,b\}\cup\{-1,1\}^w)$ for all $i$ by Lemma \ref{l74} and $y_j \in \Sigma_r\setminus(\{-a,b\}\cup\{-1,1\}^w)$ for $j \neq 1$ by Lemma \ref{l75}. Also, by Lemma \ref{l63}, $x_i \neq c-y_j$ for all $i$ and $j$. Then we have that $\{-x_i\}_{i=1}^{t+3} \cup \{c-x_i\}_{i=1}^{t+3} \cup \{y_i\}_{i=2}^t \cup \{c-y_i\}_{i=1}^t \subset \Sigma_r\setminus(\{-a,b\}\cup\{-1,1\}^w)$, which is a contradiction.

Finally, suppose that $c-x_i\neq 1$ and $y_i \neq 1$ for all $i$. Then $c-x_i \in \Sigma_r\setminus(\{-a,b\}\cup\{-1,1\}^w)$ for all $i$ by Lemma \ref{l74} and $y_j \in \Sigma_r\setminus(\{-a,b\}\cup\{-1,1\}^w)$ for all $j$ by Lemma \ref{l75}. Also, by Lemma \ref{l63}, $x_i \neq c-y_j$ for all $i$ and $j$. Then we have that
$\{-x_i\}_{i=1}^{t+3} \cup \{c-x_i\}_{i=1}^{t+3} \cup \{y_i\}_{i=1}^t \cup \{c-y_i\}_{i=1}^t \subset \Sigma_r\setminus(\{-a,b\}\cup\{-1,1\}^w)$, which is a contradiction.
\end{proof}

\section{The largest weight even case}

Let the circle act symplectically on a compact, connected symplectic manifold $M$ with exactly three fixed points. In this section, we show that if $\dim M \geq 8$, the largest weight cannot be even. We rule out case by case. In Lemma \ref{l71}, we have the following cases:
\begin{enumerate}
\item $t=0$ and $v=0$.
\item $t=0$ and $v=1$.
\item $t=0$ and $v=2$.
\item $t=1$ and $v=0$.
\item $t=1$ and $v=1$.
\item $t=1$ and $v=2$.
\item $t=2$ and $v=1$.
\item $t=2$ and $v=2$.
\item $t=3$ and $v=2$.
\item $t \geq 2+v$.
\item $v \geq 3$.
\end{enumerate}

\begin{lem} \label{l81} In Lemma \ref{l71}, $t=0$ and $v=0$ are impossible.
\end{lem}

\begin{proof} The weights in this case are
$$\Sigma_{p}=\{-c,-b,x_{1},x_{2},x_{3},1\}\cup\{-1,1\}^{w}$$
$$\Sigma_{q}=\{c,a,x_{1}-c,x_{2}-c,x_{3}-c,1\}\cup\{-1,1\}^{w}$$
$$\Sigma_{r}=\{-a,b,\cdots\}\cup\{-1,1\}^{w},$$
where $w=\min_{\alpha \in M^{S^1}} \min \{ N_{\alpha}(-1), N_{\alpha}(1) \}$, $a,b,$ and $c$ are even natural numbers such that $c=a+b$ is the largest weight, and $x_i$'s are odd natural numbers for all $i$. Moreover,  and the remaining weights at $r$ are odd.

By Lemma \ref{l76}, $x_i=c-x_j$ for some $i$ and $j$. Then by Lemma \ref{l64}, there exist $x_i$ and $x_j$ where $i \neq j$ such that $2x_i=2x_j=c$. Without loss of generality, let $2x_1=2x_2=c$. Lemma \ref{l64} also implies that $x_3 \neq c-x_i$ for all $i$. Therefore, $-x_3 \in \Sigma_r\setminus(\{-a,b\}\cup\{-1,1\}^w)$. Note that $x_1=c-x_1 =\frac{c}{2} \geq 2$.

First, suppose that $c-x_3=1$. Then we have that $x
_3=c-1>1$. Hence, $N_p(1)=N_p(-1)+1=w+1$ and $N_q(1)=N_q(-1)=w+1$. Therefore, $N_r(1)+1=N_r(-1)$ by Lemma \ref{l24} for 1. Considering Lemma \ref{l24} for each integer, one can show that the weights are
\begin{center}
$\Sigma_{p}=\{-c,-b,x_{1},x_{1},x_{3},1\}\cup\{-1,1\}^{w}$
$\Sigma_{q}=\{c,a,-x_{1},-x_{1},-1,1\}\cup\{-1,1\}^{w}$
$\Sigma_{r}=\{-a,b,-x_{3},f,-f,-1\}\cup\{-1,1\}^{w}$
\end{center}
for some odd natural number $f$. If $f>1$, by Lemma \ref{l68}, we have that $c=2x_1=2f$, which is a contradiction since no additional multiples of $x_1$ should appear by Lemma \ref{l46} part 3 for $x_1$. Hence $f=1$.

Second, suppose that $c-x_3 \neq 1$. Then by Lemma \ref{l74}, $c-x_3 \in \Sigma_r\setminus(\{-a,b\}\cup\{-1,1\}^w)$. Also, $N_p(1)\geq N_p(-1)+1=w+1$ and $N_q(1)=N_q(-1)+1=w+1$. Therefore, $N_r(1)+2\leq N_r(-1)$ by Lemma \ref{l24} for 1. Considering Lemma \ref{l24} for each integer, one can show that the weights are
\begin{center}
$\Sigma_{p}=\{-c,-b,x_{1},x_{1},x_{3},1\}\cup\{-1,1\}^{w}$
$\Sigma_{q}=\{c,a,-x_{1},-x_{1},x_{3}-c,1\}\cup\{-1,1\}^{w}$
$\Sigma_{r}=\{-a,b,-x_{3},c-x_{3},-1,-1\}\cup\{-1,1\}^{w}$
\end{center}

Therefore, in either case the weights are
\begin{center}
$\Sigma_{p}=\{-c,-b,x_{1},x_{1},x_{3},1\}\cup\{-1,1\}^{w}$
$\Sigma_{q}=\{c,a,-x_{1},-x_{1},x_{3}-c,1\}\cup\{-1,1\}^{w}$
$\Sigma_{r}=\{-a,b,-x_{3},c-x_{3},-1,-1\}\cup\{-1,1\}^{w}$
\end{center}

Let $A=c_{n}(M)|_{q}=\prod \xi_{p}^{j}, B=c_{n}(M)|_{q}=\prod \xi_{q}^{j}$, and $C=c_{n}(M)|_{r}=\prod \xi_{r}^{j}$. First, $2x_{1}=c=a+b>b$. Since $c=a+b\geq2+2=4$, $x_{1}=\frac{c}{2} \geq 2$. Therefore, $c=2x_{1}>x_{3}$ implies that $x_{1}^{2}>x_{3}$. Then
\begin{center}
$\D{(-1)^w\frac{-B-C}{a(c-x_{3})}=c x_{1}^{2}-bx_{3} > 0 },$
\end{center}
and this implies that $(-1)^{w+1}B>(-1)^wC.$ Also, we have that $(-1)^w A>0$. Then, by Theorem \ref{t21},
\begin{center}
$\D{0=(-1)^w \int_M 1 = (-1)^w \bigg( \frac{1}{A} + \frac{1}{B} + \frac{1}{C}  \bigg) >0},$
\end{center}
which is a contradiction.
\end{proof}

\begin{lem} \label{l82} In Lemma \ref{l71}, $t=0$ and $v=1$ are impossible.
\end{lem}

\begin{proof} The weights in this case are
\begin{center}
$\Sigma_{p}=\{-c,-b,x_{1},x_{2},x_{3},1\}\cup\{-1,1\}^{w+1}$
$\Sigma_{q}=\{c,a,x_{1}-c,x_{2}-c,x_{3}-c,1\}\cup\{-1,1\}^{w+1}$
$\Sigma_{r}=\{-a,b,\cdots\}\cup\{-1,1\}^{w},$
\end{center}
where $w=\min_{\alpha \in M^{S^1}} \min \{ N_{\alpha}(-1), N_{\alpha}(1) \}$, $a,b,$ and $c$ are even natural numbers such that $c=a+b$ is the largest weight, and $x_i$'s are odd natural numbers for all $i$. Moreover,  and the remaining weights at $r$ are odd.
\begin{enumerate}
\item $x_i \neq c-x_j$ for all $i$ and $j$.

By Lemma \ref{l72}, $-x_i \in \Sigma_r\setminus(\{-a,b\}\cup\{-1,1\}^w)$ for all $i$. Assume that $x_i=x_j$ for some $i\neq j$. Then by Lemma \ref{l66}, $2x_i=2x_j=c$ hence $x_i=c-x_j$, which contradicts the assumption. Hence $x_i \neq x_j$ for $i \neq j$.

First, assume that $c-x_i \neq 1$ for all $i$. Then by Lemma \ref{l74}, $c-x_i \in \Sigma_r\setminus(\{-a,b\}\cup\{-1,1\}^w)$ for all $i$. Hence, the weights are
\begin{center}
$\Sigma_{p}=\{-c,-b,x_{1},x_{2},x_{3},1,-1,1\}\cup\{-1,1\}^{w}$
$\Sigma_{q}=\{c,a,x_{1}-c,x_{2}-c,x_{3}-c,1,-1,1\}\cup\{-1,1\}^{w}$
$\Sigma_{r}=\{-a,b,-x_1,-x_2,-x_3,c-x_1,c-x_2,c-x_3\}\cup\{-1,1\}^{w}$
\end{center}
Then we have that $\lambda_r=\frac{1}{2}\dim M$, which contradicts Lemma \ref{l33} that $\lambda_r=\frac{1}{2}\dim M+2$.

Next, assume that $c-x_i=1$ for some $i$. Without loss of generality, let $c-x_3=1$. Then by Lemma \ref{l66}, $c-x_i \neq 1$ for $i \neq 3$. By Lemma \ref{l74}, $c-x_i \in \Sigma_r\setminus(\{-a,b\}\cup\{-1,1\}^w)$ for $i \neq 3$. Also, by the assumption, $x_i \neq 1$ for $i \neq 3$. Then $N_p(1)= N_p(-1)+1=w+2$ and $N_q(1)=N_q(-1)=w+2$. Therefore, $N_r(1)+1= N_r(-1)$ by Lemma \ref{l24} for 1. Then the weights are
\begin{center}
$\Sigma_{p}=\{-c,-b,x_{1},x_{2},c-1,1,-1,1\}\cup\{-1,1\}^{w}$
$\Sigma_{q}=\{c,a,x_{1}-c,x_{2}-c,-1,1,-1,1\}\cup\{-1,1\}^{w}$
$\Sigma_{r}=\{-a,b,-x_{1},-x_{2},1-c,c-x_{1},c-x_{2},-1\}\cup\{-1,1\}^{w}$
\end{center}
By Lemma \ref{l46} part 1 for $x_3=c-1$, $\Sigma_p \equiv \Sigma_r \mod c-1$. Since $c-x_3=1$, $c-x_1 \neq 1$ and $c-x_2 \neq 2$ by Lemma \ref{l66}. Hence $N_r(1)=w$. Then $N_p(1)\geq w+2$, $N_r(1)=w$, and $\Sigma_p \equiv \Sigma_r \mod c-1$ imply that $N_r(2-c)\geq 2$ since $|1+k(c-1)| > c|$ for $|k|\geq 2$ and $c \notin \Sigma_r$. However, $r$ has only one negative even weight, which is a contradiction.

\item $x_{i}= c-x_{j}$ for some $i$ and $j$.

By Lemma \ref{l76}, $x_i=c-x_j$ for some $i$ and $j$. Then by Lemma \ref{l64}, there exist $x_i$ and $x_j$ where $i \neq j$ such that $2x_i=2x_j=c$. Without loss of generality, let $2x_1=2x_2=c$. Lemma \ref{l64} also implies that $x_3 \neq c-x_i$ for all $i$. Therefore, by Lemma \ref{l72}, $-x_3 \in \Sigma_r\setminus(\{-a,b\}\cup\{-1,1\}^w)$. Note that $x_1=c-x_1 =\frac{c}{2} \geq 2$.

First, assume that $c-x_3=1$. Then $N_p(1)=N_p(-1)+1=w+2$ and $N_q(1)=N_q(-1)=w+2$. Therefore, $N_r(1)+1=N_r(-1)$ by Lemma \ref{l24} for 1. Considering Lemma \ref{l24} for each integer, one can show that the weights are
\begin{center}
$\Sigma_{p}=\{-c,-b,x_{1},x_{1},c-1,1,-1,1\}\cup\{-1,1\}^{w}$
$\Sigma_{q}=\{c,a,-x_1,-x_1,-1,1,-1,1\}\cup\{-1,1\}^{w}$
$\Sigma_{r}=\{-a,b,1-c,-1,-f,f,-g,g\}\cup\{-1,1\}^{w}$
\end{center}
for some odd natural numbers $f$ and $g$. If $f>1$, then by Lemma \ref{l68}, $\{-2f,-f\}\subset\Sigma_p$, which is a contradiction since $p$ has no negative odd weight that is less than -1. Hence $f=1$. However, this means that $\min_{\alpha \in M^{S^1}} \min \{ N_{\alpha}(-1), N_{\alpha}(1) \} \geq w+1$, which is a contradiction.

Second, assume that $c-x_3 \neq 1$. By Lemma \ref{l74}, $c-x_3 \in \Sigma_r\setminus(\{-a,b\}\cup\{-1,1\}^w)$. Then $N_p(1)\geq N_p(-1)+1=w+2$ and $N_q(1)=N_q(-1)+1=w+2$. Therefore, $N_r(1)+2 \leq N_r(-1)$ by Lemma \ref{l24} for 1. Considering Lemma \ref{l24} for each integer, one can show that the weights are
\begin{center}
$\Sigma_{p}=\{-c,-b,x_{1},x_{2},x_{3},1,-1,1\}\cup\{-1,1\}^{w}$
$\Sigma_{q}=\{c,a,-x_2,-x_1,x_{3}-c,1,-1,1\}\cup\{-1,1\}^{w}$
$\Sigma_{r}=\{-a,b,-x_{3},c-x_{3},-1,-1,-f,f\}\cup\{-1,1\}^{w}$
\end{center}
for some odd natural number $f$. As above, $f=1$ and this means that $\min_{\alpha \in M^{S^1}} \min \{ N_{\alpha}(-1), N_{\alpha}(1) \} \geq w+1$, which is a contradiction.
\end{enumerate}
\end{proof}

\begin{lem} \label{l83} In Lemma \ref{l71}, $t=0$ and $v=2$ are impossible.
\end{lem}

\begin{proof} The weights in this case are
\begin{center}
$\Sigma_{p}=\{-c,-b,x_{1},x_{2},x_{3},1\}\cup\{-1,1\}^{w+2}$
$\Sigma_{q}=\{c,a,x_{1}-c,x_{2}-c,x_{3}-c,1\}\cup\{-1,1\}^{w+2}$
$\Sigma_{r}=\{-a,b,\cdots\}\cup\{-1,1\}^{w},$
\end{center}
where $w=\min_{\alpha \in M^{S^1}} \min \{ N_{\alpha}(-1), N_{\alpha}(1) \}$, $a,b,$ and $c$ are even natural numbers such that $c=a+b$ is the largest weight, and $x_i$'s are odd natural numbers for all $i$. Moreover,  and the remaining weights at $r$ are odd.

\begin{enumerate}
\item $x_i \neq c-x_j$ for all $i$ and $j$.

By Lemma \ref{l72}, $-x_i \in \Sigma_r\setminus(\{-a,b\}\cup\{-1,1\}^w)$ for all $i$. Assume that $x_i=x_j$ for some $i\neq j$. Then by Lemma \ref{l66}, $2x_i=2x_j=c$ hence $x_i=c-x_j$, which contradicts the assumption. Hence $x_i \neq x_j$ for $i \neq j$.

First, assume that $c-x_i \neq 1$ for all $i$. Then by Lemma \ref{l74}, $c-x_i \in \Sigma_r\setminus(\{-a,b\}\cup\{-1,1\}^w)$ for all $i$. Also, $N_p(1)\geq N_p(-1)+1=w+3$ and $N_q(1)=N_q(-1)+1=w+3$. Therefore, $N_r(1)+2 \leq N_r(-1)$ by Lemma \ref{l24} for 1. Considering Lemma \ref{l24} for each integer, one can show that the weights are
\begin{center}
$\Sigma_{p}=\{-c,-b,x_{1},x_{2},x_{3},1,-1,1,-1,1\}\cup\{-1,1\}^{w}$
$\Sigma_{q}=\{c,a,x_{1}-c,x_{2}-c,x_{3}-c,1,-1,1,-1,1\}\cup\{-1,1\}^{w}$
$\Sigma_{r}=\{-a,b,-x_1,-x_2,-x_3,c-x_1,c-x_2,c-x_3,-1,-1\}\cup\{-1,1\}^{w}$
\end{center}
Then we have that $N_{p}(1) \geq w+3$, $N_{q}(1) \geq w+3$, and $N_{r}(1) =w$, which contradict Lemma \ref{l69}.

Next, assume that $c-x_i=1$ for some $i$. Without loss of generality, let $c-x_3=1$. Then by Lemma \ref{l66}, $c-x_i \neq 1$ for $i \neq 3$. Hence, by Lemma \ref{l74}, $c-x_i \in \Sigma_r\setminus(\{-a,b\}\cup\{-1,1\}^w)$ for $i \neq 3$. Also, $N_p(1)\geq N_p(-1)+1=w+3$ and $N_q(1)=N_q(-1)=w+3$. Therefore, $N_r(1)+1 \leq N_r(-1)$ by Lemma \ref{l24} for 1. Considering Lemma \ref{l24} for each integer, one can show that the weights are
\begin{center}
$\Sigma_{p}=\{-c,-b,x_{1},x_{2},c-1,1,-1,1,-1,1\}\cup\{-1,1\}^{w}$
$\Sigma_{q}=\{c,a,x_{1}-c,x_{2}-c,-1,1,-1,1,-1,1\}\cup\{-1,1\}^{w}$
$\Sigma_{r}=\{-a,b,-x_{1},-x_{2},1-c,c-x_{1},c-x_{2},-1,-f,f\}\cup\{-1,1\}^{w}$
\end{center}
for some odd natural number $f$. If $f>1$, then by Lemma \ref{l68}, $\{-2f,-f\}\subset\Sigma_p$, which is a contradiction since $p$ has no negative odd weight that is less than -1. Hence $f=1$. However, this means that $\min_{\alpha \in M^{S^1}} \min \{ N_{\alpha}(-1), N_{\alpha}(1) \} \geq w+1$, which is a contradiction.

\item $x_{i}= c-x_{j}$ for some $i$ and $j$.

By Lemma \ref{l64}, there exist $x_i$ and $x_j$ where $i \neq j$ such that $2x_i=2x_j=c$. Without loss of generality, let $2x_1=2x_2=c$. Lemma \ref{l64} also implies that $x_3 \neq c-x_i$ for all $i$. Therefore, $-x_3 \in \Sigma_r\setminus(\{-a,b\}\cup\{-1,1\}^w)$. Note that $x_1=c-x_1 =\frac{c}{2} \geq 2$.

First, assume that $c-x_3=1$. Then we have that $N_p(1)= N_p(-1)+1=w+3$ and $N_q(1)=N_q(-1)=w+3$. Therefore, $N_r(1)+1 = N_r(-1)$ by Lemma \ref{l24} for 1. Considering Lemma \ref{l24} for each integer, one can show that the weights are
\begin{center}
$\Sigma_{p}=\{-c,-b,x_{1},x_{1},c-1,1,-1,1,-1,1\}\cup\{-1,1\}^{w}$
$\Sigma_{q}=\{c,a,-x_1,-x_1,-1,1,-1,1,-1,1\}\cup\{-1,1\}^{w}$
$\Sigma_{r}=\{-a,b,1-c,-1,-f,f,-h,h,-k,k\}\cup\{-1,1\}^{w}$
\end{center}
for some odd natural numbers $f,h$, and $k$. As above, $f=1$ and this means that $\min_{\alpha \in M^{S^1}} \min \{ N_{\alpha}(-1), N_{\alpha}(1) \} \geq w+1$, which is a contradiction.

Second, assume that $c-x_3 \neq 1$. By Lemma \ref{l74}, $c-x_3 \in \Sigma_r\setminus(\{-a,b\}\cup\{-1,1\}^w)$. Also, $N_p(1)\geq N_p(-1)+1=w+3$ and $N_q(1)= N_q(-1)+1=w+3$. Therefore, $N_r(1)+2 \leq N_r(-1)$ by Lemma \ref{l24} for 1. Considering Lemma \ref{l24} for each integer, one can show that the weights are
\begin{center}
$\Sigma_{p}=\{-c,-b,x_{1},x_{1},x_{3},1,-1,1,-1,1\}\cup\{-1,1\}^{w}$
$\Sigma_{q}=\{c,a,-x_1,-x_1,x_{3}-c,1,-1,1,-1,1\}\cup\{-1,1\}^{w}$
$\Sigma_{r}=\{-a,b,-x_{3},c-x_{3},-1,-1,-f,f,-h,h\}\cup\{-1,1\}^{w}$
\end{center}
for some odd natural numbers $f$ and $h$. As above, $f=1$ and this means that $\min_{\alpha \in M^{S^1}} \min \{ N_{\alpha}(-1), N_{\alpha}(1) \} \geq w+1$, which is a contradiction.
\end{enumerate}
\end{proof}

\begin{lem} \label{l84} In Lemma \ref{l71}, $t=1$ and $v=0$ are impossible.
\end{lem}

\begin{proof}
The weights in this case are
\begin{center}
$\Sigma_{p}=\{-c,-b,x_{1},x_{2},x_{3},x_{4},-y,1\}\cup\{-1,1\}^{w}$
$\Sigma_{q}=\{c,a,x_{1}-c,x_{2}-c,x_{3}-c,x_{4}-c,c-y,1\}\cup\{-1,1\}^{w}$
$\Sigma_{r}=\{-a,b,\cdots\}\cup\{-1,1\}^{w},$
\end{center}
where $a,b,$ and $c$ are even natural numbers such that $c=a+b$ is the largest weight, $x_i$'s and $y$ are odd natural numbers for all $i$, and $w=\min_{\alpha \in M^{S^1}} \min \{ N_{\alpha}(-1), N_{\alpha}(1) \}$. Moreover, the remaining weights at $r$ are odd.

By Lemma \ref{l76}, $x_i=c-x_j$ for some $i$ and $j$. Then by Lemma \ref{l64}, there exist $x_i$ and $x_j$ where $i \neq j$ such that $2x_i=2x_j=c$. Without loss of generality, let $2x_1=2x_2=c$. Lemma \ref{l64} also implies that $x_i \neq c-x_j$ for $i \neq 1$ and 2, and for all $j$. Therefore, $-x_i \in \Sigma_r\setminus(\{-a,b\}\cup\{-1,1\}^w)$ for $i =3$ and 4 by Lemma \ref{l72}. Also, by Lemma \ref{l73}, $y-c \in \Sigma_r\setminus(\{-a,b\}\cup\{-1,1\}^w)$. Moreover, by Lemma \ref{l46} part 3 for $x$, none of $x_i$'s, $y$, $c-x_i$'s and $c-y$ can be $x$ for $i \neq 1$ and 2. Hence, by Lemma \ref{l66}, all of $x_i$'s, $y$, $c-x_i$'s and $c-y$ are different for $i \neq 1$ and 2.

First, suppose that $c-x_i \neq 1$ for all $i$ and $y \neq 1$. Then by Lemma \ref{l74}, $x_i-c \in \Sigma_r\setminus(\{-a,b\}\cup\{-1,1\}^w)$ for $i=3$ and 4. Also, by Lemma \ref{l75}, $y \in \Sigma_r\setminus(\{-a,b\}\cup\{-1,1\}^w)$. Then the weights are
\begin{center}
$\Sigma_{p}=\{-c,-b,x_{1},x_{1},x_{3},x_{4},-y,1\}\cup\{-1,1\}^{w}$
$\Sigma_{q}=\{c,a,-x_{1},-x_{1},x_{3}-c,x_{4}-c,c-y,1\}\cup\{-1,1\}^{w}$
$\Sigma_{r}=\{-a,b,-x_{3},-x_{4},y,c-x_{3},c-x_{4},y-c\}\cup\{-1,1\}^{w}$
\end{center}
Then we have that $\lambda_r=\frac{1}{2}\dim M$, which contradicts Lemma \ref{l33} that $\lambda_r=\frac{1}{2}\dim M+2$.

Second, suppose that $y=1$. Then by Lemma \ref{l66}, none of $x_i$'s, $c-x_i$'s, and $c-y$ is 1 for all $i$. Hence, by Lemma \ref{l74}, $x_i-c \in \Sigma_r\setminus(\{-a,b\}\cup\{-1,1\}^w)$ for $i = 3$ and 4. Moreover, $N_p(1)=N_p(-1)=w+1$ and $N_q(1)-1=N_q(-1)=w$. By Lemma \ref{l24} for 1, this implies that $N_r(1)=N_r(-1)-1$. Then the weights are
\begin{center}
$\Sigma_{p}=\{-c,-b,x_{1},x_{1},x_{3},x_{4},-1,1\}\cup\{-1,1\}^{w}$
$\Sigma_{q}=\{c,a,-x_{1},-x_{1},x_{3}-c,x_{4}-c,c-1,1\}\cup\{-1,1\}^{w}$
$\Sigma_{r}=\{-a,b,-x_{3},-x_{4},c-x_{3},c-x_{4},1-c,-1\}\cup\{-1,1\}^{w}$
\end{center}
By Corollary \ref{c27}, $c_{1}(M)|_{p}=0$ and this implies that $x_{3}+x_{4}=b$. Since $x_{3}+x_{4}=b<a+b=c=2x_1$ and $x_1 = \frac{c}{2} \geq2$, we have that $x_{1}^{2} > x_{3}x_{4}$. Therefore,
\begin{center}
$\D{(-1)^w\frac{B+C}{a(c-x_{3})(c-x_{4})(c-1)}=c x_{1}^{2}-bx_{3}x_{4}>0}.$
\end{center}
Also, $(-1)^w A<0$. Then, by Theorem \ref{t21},
\begin{center}
$\D{0=\int_M 1 = (-1)^w \bigg(\frac{1}{A} + \frac{1}{B} + \frac{1}{C} \bigg) <0},$
\end{center}
which is a contradiction.

Finally, suppose that $c-x_{i}=1$ for some $i$. Since $c \geq 4$, $c-x_1=x_1=\frac{c}{2} \geq 2$. Hence, $c-x_i=x_i \neq 1$ for $i=1$ and 2. Therefore, without loss of generality, let $c-x_{4}=1$. By Lemma \ref{l66}, none of $x_i$'s, $c-x_j$'s, $y$, and $c-y$ is 1 for all $i$ and for $j \neq 4$. Hence, by Lemma \ref{l74}, $c-x_3 \in \Sigma_r\setminus(\{-a,b\}\cup\{-1,1\}^w)$. Also, by Lemma \ref{l75}, $y \in \Sigma_r\setminus(\{-a,b\}\cup\{-1,1\}^w)$. Moreover, $N_p(1)=N_p(-1)+1=w+1$ and $N_q(1)=N_q(-1)=w+1$. By Lemma \ref{l24} for 1, this implies that $N_r(1)=N_r(-1)-1$. Then the weights are
\begin{center}
$\Sigma_{p}=\{-c,-b,x_{1},x_{1},x_{3},c-1,-y,1\}\cup\{-1,1\}^{w}$
$\Sigma_{q}=\{c,a,-x_{1},-x_{1},x_{3}-c,-1,c-y,1\}\cup\{-1,1\}^{w}$
$\Sigma_{r}=\{-a,b,-x_{3},1-c,c-x_{3},y-c,y,-1\}\cup\{-1,1\}^{w}$
\end{center}
By Lemma \ref{l46} part 1 for $c-1=x_4$, $\Sigma_{p}\equiv\Sigma_{r}\mod c-1$. First, we can choose a bijection between $\Sigma_p$ and $\Sigma_q$ so that $\Sigma_{p}\supset\{c-1,-c\} \cup \{-1,1\}^w\equiv \{1-c,-1\}\cup \{-1,1\}^w\subset\Sigma_{r} \mod c-1$. Since $N_p(1)=w+1$ and $N_r(1)=w$, $\Sigma_{p}\equiv\Sigma_{r}\mod c-1$ implies that $2-c \in \Sigma_r$ since $|1+k(c-1)| > c$ for $|k| \geq 2$ and $c \notin \Sigma_r$. Since $-a$ is the only negative even weight at $r$, we have that $-a=2-c$, i.e., $a+2=c=a+b$. Hence, $b=2$. Then we are left with
\begin{center}
$\{-2=-b,x_{1},x_{1},x_{3},-y\}$
$\equiv \{2=b,-x_{3},c-x_{3},y-c,y\} \mod c-1.$
\end{center}
Since for $2 \in \Sigma_r$, $2 \neq -2, x_{1}$, and $x_{3} \mod c-1$, the only possibility is that $2 \equiv -y \mod c-1$, i.e., $2-c+1=-y$. Thus $y=c-3$. By Corollary \ref{c27}, $c_{1}(M)|_{p}=-c-2+x_{1}+x_{1}+x_{3}+c-1-y+1=0$. Hence, we have that $x_3+c=y+2$. However, $x_{3}+c=y+2=c-3+2=c-1$ and so $0<x_{3}=-1$, which is a contradiction.
\end{proof}

\begin{lem} \label{l85} In Lemma \ref{l71}, $t=1$ and $v=1$ are impossible.
\end{lem} 

\begin{proof} The weights in this case are
\begin{center}
$\Sigma_{p}=\{-c,-b,x_{1},x_{2},x_{3},x_{4},-y,1\}\cup\{-1,1\}^{w+1}$
$\Sigma_{q}=\{c,a,x_{1}-c,x_{2}-c,x_{3}-c,x_{4}-c,c-y,1\}\cup\{-1,1\}^{w+1}$
$\Sigma_{r}=\{-a,b,\cdots\}\cup\{-1,1\}^{w},$
\end{center}
where $a,b,$ and $c$ are even natural numbers such that $c=a+b$ is the largest weight, $x_i$'s and $y$ are odd natural numbers for all $i$, and $w=\min_{\alpha \in M^{S^1}} \min \{ N_{\alpha}(-1), N_{\alpha}(1) \}$. Moreover, the remaining weights at $r$ are odd.

By Lemma \ref{l76}, $x_i=c-x_j$ for some $i$ and $j$. Then by Lemma \ref{l64}, there exist $x_i$ and $x_j$ where $i \neq j$ such that $2x_i=2x_j=c$. Without loss of generality, let $2x_1=2x_2=c$. Denote $x=x_1$. Lemma \ref{l64} also implies that $x_i \neq c-x_j$ for $i \neq 1$ and 2, and for all $j$. Therefore, $-x_i \in \Sigma_r\setminus(\{-a,b\}\cup\{-1,1\}^w)$ for $i=3$ and 4 by Lemma \ref{l72}. Also, by Lemma \ref{l73}, $y-c \in \Sigma_r\setminus(\{-a,b\}\cup\{-1,1\}^w)$. Moreover, by Lemma \ref{l46} part 3 for $x$, none of $x_i$'s, $y$, $c-x_i$'s, and $c-y$ can be $x$ for $i \neq 1$ and 2. Hence, by Lemma \ref{l66}, all of $x_i$'s, $y$, $c-x_i$'s, and $c-y$ are different for $i \neq 1$ and 2. We have the following cases:

\begin{enumerate}
\item $y=1$.

By Lemma \ref{l66}, none of $x_i$'s, $c-x_i$'s, and $c-y$ is 1 for all $i$. Then by Lemma \ref{l74}, $c-x_i \in \Sigma_r\setminus(\{-a,b\}\cup\{-1,1\}^w)$ for $i=3$ and 4. Moreover, $N_p(1)=N_p(-1)=N_q(1)=N_q(-1)+1=w+2$. Hence, by Lemma \ref{l24} for 1, $N_r(-1)=N_r(1)+1$. Considering Lemma \ref{l24} for each integer, one can show that the weights are
\begin{center}
$\Sigma_{p}=\{-c,-b,x,x,x_{3},x_{4},-1,1,-1,1\}\cup\{-1,1\}^{w}$
$\Sigma_{q}=\{c,a,-x,-x,x_{3}-c,x_{4}-c,c-1,1,-1,1\}\cup\{-1,1\}^{w}$
$\Sigma_{r}=\{-a,b,-x_{3},-x_{4},c-x_{3},c-x_{4},1-c,-1,-f,f\}\cup\{-1,1\}^{w}$
\end{center}
for some odd natural number $f$. Suppose that $f>1$. Then by Lemma \ref{l68}, $c=2x=2f$, which is a contradiction since no additional multiples of $x$ should appear by Lemma \ref{l46} part 3 for $x$. Hence $f=1$. However, this means that $\min_{\alpha \in M^{S^1}} \min \{ N_{\alpha}(-1), N_{\alpha}(1) \} \geq w+1$, which is a contradiction.

\item $c-x_{i}=1$ for some $i$.

Since $2x_1=2x_2=c\geq4$, $c-x_1=c-x_2=x_1 \geq 2$. Hence, without loss of generality, assume that $c-x_{4}=1$. By Lemma \ref{l66}, none of $x_i$'s, $y$, $c-x_j$'s, and $c-y$ can be 1 for all $i$ and $j \neq 4$. Thus $y \in \Sigma_r\setminus(\{-a,b\}\cup\{-1,1\}^w)$ by Lemma \ref{l75} and $c-x_3 \in \Sigma_r\setminus(\{-a,b\}\cup\{-1,1\}^w)$ by Lemma \ref{l74}. Moreover, $N_p(1)=N_p(-1)+1=N_q(1)=N_q(-1)=w+2$. Hence, by Lemma \ref{l24} for 1, $N_r(-1)=N_r(1)+1$. Considering Lemma \ref{l24} for each integer, one can show that the weights are
\begin{center}
$\Sigma_{p}=\{-c,b,x,x,x_{3},c-1,-y,1,-1,1\}\cup\{-1,1\}^{w}$
$\Sigma_{q}=\{c,a,-x,-x,x_{3}-c,-1,c-y,1,-1,1\}\cup\{-1,1\}^{w}$
$\Sigma_{r}=\{-a,b,-x_{3},1-c,y,c-x_{3},y-c,-1,-f,f\}\cup\{-1,1\}^{w}$
\end{center}
for some odd natural number $f$. Then as above, $f=1$, which is a contradiction since this means that $\min_{\alpha \in M^{S^1}} \min \{ N_{\alpha}(-1), N_{\alpha}(1) \} \geq w+1$.

\item $y \neq 1$ and $c-x_{i} \neq1$, for all $i$.

By Lemma \ref{l75}, $y \in \Sigma_r\setminus(\{-a,b\}\cup\{-1,1\}^w)$. Also, $c-x_i \in \Sigma_r\setminus(\{-a,b\}\cup\{-1,1\}^w)$ for $i=3$ and 4 by Lemma \ref{l74}. Moreover, $N_p(1)\geq N_p(-1)+1$ and $N_q(1) \geq N_q(-1)+1$. Hence, by Lemma \ref{l24} for 1, $N_r(-1) \geq N_r(1)+2$. Considering Lemma \ref{l24} for each integer, one can show that the weights are
\begin{center}
$\Sigma_{p}=\{-c,-b,x,x,x_{3},x_{4},-y,1,-1,1\}\cup\{-1,1\}^{w}$
$\Sigma_{q}=\{c,a,-x,-x,x_{3}-c,x_{4}-c,c-y,1,-1,1\}\cup\{-1,1\}^{w}$
$\Sigma_{r}=\{-a,b,-x_{3},-x_{4},c-x_{3},c-x_{4},y,y-c,-1,-1\}\cup\{-1,1\}^{w}$
\end{center}

By Lemma \ref{l46} part 3 for $x$, no additional multiples of $x$ appear. Hence, $a \neq x$ and $b \neq x$. Since $2x=c=a+b$, either $a > \frac{c}{2}$ or $b>\frac{c}{2}$.

\begin{enumerate}[(a)]
\item $a > \frac{c}{2}$.

By Lemma \ref{l46} part 1 for $a$, $\Sigma_{q} \equiv \Sigma_{r} \mod a$. First, we can choose so that 
\begin{center}
$\Sigma_q \supset \{c=a+b,a,-1\}\cup\{-1,1\}^w$
$\equiv \{-a,b,-1\}\cup\{-1,1\}^w \subset \Sigma_r \mod a.$
\end{center}
Next, $\{-x,-x\}\subset\Sigma_{q}$ and $-x \notin \Sigma_{r}$ imply that $\{a-x,a-x\}\subset\Sigma_{r}\setminus(\{-a,b,-1\}\cup\{-1,1\}^w)$, $\{2a-x,2a-x\}\subset\Sigma_{r}\setminus(\{-a,b,-1\}\cup\{-1,1\}^w)$, or $\{a-x,2a-x\}\subset\Sigma_{r}\setminus(\{-a,b,-1\}\cup\{-1,1\}^w)$, since $|-x+ka|>c$ for $k<0$ or $k>2$. If the first case or the second case holds, it implies that two of $c-x_3$, $c-x_4$, and $y$ are equal since $c-x_3$, $c-x_4$, and $y$ are the only positive integers in $\Sigma_{r}\setminus(\{-a,b,-1\}\cup\{-1,1\}^w)$, which is a contradiction by Lemma \ref{l66}. Hence we have that $\{a-x,2a-x\}\subset\Sigma_{r}\setminus(\{-a,b,-1\}\cup\{-1,1\}^w)$.

Similarly, $N_q(1)=w+2$, $N_r(1)=w$, and $\Sigma_{q} \equiv \Sigma_{r} \mod a$ imply that either $\{1+a,1+a\}\subset\Sigma_{r}\setminus(\{-a,b,-1\}\cup\{-1,1\}^w)$, $\{1-a,1-a\}\subset\Sigma_{r}\setminus(\{-a,b,-1\}\cup\{-1,1\}^w)$, or $\{1-a,1+a\}\subset\Sigma_{r}\setminus(\{-a,b,-1\}\cup\{-1,1\}^w)$, since $|1+ka|>c$ for $|k| \geq 2$. If the first case holds, $-2(a+1) \in \Sigma_r$ by Lemma \ref{l46} part 3 for $a+1$, which is a contradiction since $-2(a+1) <-c$ where $-c$ is the smallest weight. If the second case holds, $2(a-1) \in \Sigma_r$ by Lemma \ref{l46} part 3 for $1-a$, which is a contradiction since $2(a-1) \geq 2x=c$ but $c \notin \Sigma_r$. Hence, the third case must be the case.

To sum up, we have that $\{1-a,1+a,a-x,2a-x\}\subset\{-x_3,-x_4,y,c-x_3,c-x_4,y-c\}=\Sigma_r\setminus(\{-a,b,-1\}\cup\{-1,1\}^w)$, i.e., $1-a \in \{-x_3,-x_4,y-c\}$ and $\{1+a,a-x,2a-x\} = \{y,c-x_3,c-x_4\}$. For each $\alpha \in \{-x_3,-x_4,y-c\}$, we have that $\alpha +c \in \{y,c-x_3,c-x_4\}$. This implies that $1-a+c \in \{y,c-x_3,c-x_4\}=\{1+a,a-x,2a-x\}$. If $1-a+c=1+a$, then $1-a+c=1-a+2x=1+a$, hence $2x=2a$, which contradicts that $a>x$. If $1-a+c=2a-x$, then $1-a+c=1-a+2x=2a-x$, hence $3x+1=3a$, which is a contradiction since $a>x$. Hence $1-a+c=a-x$. Then we have that $c+1=2a-x$, which is a contradiction since $2a-x \in \{y,c-x_3,c-x_4\}\subset\Sigma_r$ but $2a-x=c+1>c$, where $c$ is the largest weight.

\item $b > \frac{c}{2}$.

By Lemma \ref{l46} part 1 for $b$, $\Sigma_{p} \equiv \Sigma_{r} \mod b$. First, we can choose so that 
\begin{center}
$\Sigma_p \supset \{-c=-a-b,-b,-1\}\cup\{-1,1\}^w$
$\equiv \{-a,b,-1\}\cup\{-1,1\}^w \subset \Sigma_r \mod b.$
\end{center}
Next, as above, one can show that $\{b-x,2b-x,b+1\}=\{y,c-x_3,c-x_4\}$. This implies that $b-x+2b-x+b+1=4b-2x+1=2c+y-x_3-x_4=y+c-x_3+c-x_4$. Therefore, $4b-2x+1=2c+y-x_3-x_4=4x+y-x_3-x_4$, hence $4b+1=6x+y-x_3-x_4$. On the other hand, by Corollary \ref{c27}, $c_1(M)|_p=b+x_3+x_4-y+1=0$. Thus $y-x_3-x_4=b+1$. Then we have that $4b+1=6x+y-x_3-x_4=6x+b+1$, hence $3b=6x$, which is a contradiction since $b<2x=c$.
\end{enumerate}
\end{enumerate}
\end{proof}

\begin{lem} \label{l86} In Lemma \ref{l71}, $t=1$ and $v=2$ are impossible.
\end{lem}

\begin{proof} The weights in this case are
\begin{center}
$\Sigma_{p}=\{-c,-b,x_{1},x_{2},x_{3},x_{4},-y,1\}\cup\{-1,1\}^{w+2}$
$\Sigma_{q}=\{c,a,x_{1}-c,x_{2}-c,x_{3}-c,x_{4}-c,c-y,1\}\cup\{-1,1\}^{w+2}$
$\Sigma_{r}=\{-a,b,\cdots\}\cup\{-1,1\}^{w},$
\end{center}
where $a,b,$ and $c$ are even natural numbers such that $c=a+b$ is the largest weight, $x_i$'s and $y$ are odd natural numbers for all $i$, and $w=\min_{\alpha \in M^{S^1}} \min \{ N_{\alpha}(-1), N_{\alpha}(1) \}$. Moreover, the remaining weights at $r$ are odd. By Lemma \ref{l73}, $y-c \in \Sigma_r\setminus(\{-a,b\}\cup\{-1,1\}^w)$. We have the following cases:

\begin{enumerate}
\item $x_i=c-x_j$ for some $i$ and $j$.

By Lemma \ref{l64}, there exist $x_i$ and $x_j$ where $i \neq j$ such that $2x_i=2x_j=c$. Without loss of generality, let $2x_1=2x_2=c$. Denote $x=x_1$. Lemma \ref{l64} also implies that $x_i \neq c-x_j$ for $i \neq 1$ and 2, and for all $j$. Therefore, $-x_i \in \Sigma_r\setminus(\{-a,b\}\cup\{-1,1\}^w)$ for $i =3$ and 4 by Lemma \ref{l72}. Moreover, by Lemma \ref{l46} part 3 for $x$, none of $x_i$'s, $y$, $c-x_i$'s, and $c-y$ can be $x$ for $i \neq 1$ and 2. Hence, by Lemma \ref{l66}, all of $x_i$'s, $y$, $c-x_i$'s, and $c-y$ are different for $i \neq 1$ and 2. 

First, assume that $y=1$. Then none of $x_i$'s, $c-x_i$'s, and $c-y$ is 1 for all $i$. Hence $c-x_i \in \Sigma_r\setminus(\{-a,b\}\cup\{-1,1\}^w)$ for $i=3$ and 4 by Lemma \ref{l74}. Moreover, $N_p(1)=N_p(-1)=N_q(1)=N_q(-1)+1=w+3$. Hence, by Lemma \ref{l24} for 1, $N_r(-1)=N_r(1)+1$. Considering Lemma \ref{l24} for each integer, one can show that the weights are
\begin{center}
$\Sigma_{p}=\{-c,-b,x,x,x_{3},x_{4},-1,1,-1,1,-1,1\}\cup\{-1,1\}^{w}$
$\Sigma_{q}=\{c,a,-x,-x,x_{3}-c,x_{4}-c,c-1,1,-1,1,-1,1\}\cup\{-1,1\}^{w}$
$\Sigma_{r}=\{-a,b,-x_3,-x_4,c-x_3,c-x_4,1-c,-1,-f,f,-h,h\}\cup\{-1,1\}^{w}$
\end{center}
for some odd natural numbers $f$ and $h$. Suppose that $f>1$. Then by Lemma \ref{l68}, $c=2x=2f$, which contradicts Lemma \ref{l46} part 3 for $x$. Hence $f=1$, which is a contradiction since it means that $\min_{\alpha \in M^{S^1}} \min \{ N_{\alpha}(-1), N_{\alpha}(1) \} \geq w+1$.

Second, assume that $c-x_i=1$ for some $i$. Since $c-x_1=c-x_2=x_1=\frac{c}{2}\geq2$, without loss of generality, let $c-x_{4}=1$. Then none of $x_i$'s, $c-x_j$'s, $y$, and $c-y$ is 1 for all $i$ and $j \neq 4$. Thus $y \in \Sigma_r\setminus(\{-a,b\}\cup\{-1,1\}^w)$ by Lemma \ref{l75} and $c-x_3 \in \Sigma_r\setminus(\{-a,b\}\cup\{-1,1\}^w)$ by Lemma \ref{l74}. Moreover, $N_p(1)=N_p(-1)+1=N_q(1)=N_q(-1)=w+3$. Hence, by Lemma \ref{l24} for 1, $N_r(-1)=N_r(1)+1$. Considering Lemma \ref{l24} for each integer, one can show that the weights are
\begin{center}
$\Sigma_{p}=\{-c,-b,x,x,x_{3},c-1,-y,1,-1,1,-1,1\}\cup\{-1,1\}^{w}$
$\Sigma_{q}=\{c,a,-x,-x,x_{3}-c,-1,c-y,1,-1,1,-1,1\}\cup\{-1,1\}^{w}$
$\Sigma_{r}=\{-a,b,-x_{3},1-c,y,c-x_{3},y-c,-1,-f,f,-h,h\}\cup\{-1,1\}^{w}$
\end{center}
for some odd natural numbers $f$ and $h$. As above, $f=1$, which is a contradiction since it means that $\min_{\alpha \in M^{S^1}} \min \{ N_{\alpha}(-1), N_{\alpha}(1) \} \geq w+1$.

Last, assume that $y \neq 1$ and $c-x_i \neq 1$ for all $i$. By Lemma \ref{l75}, $y \in \Sigma_r\setminus(\{-a,b\}\cup\{-1,1\}^w)$. Also, $c-x_i \in \Sigma_r\setminus(\{-a,b\}\cup\{-1,1\}^w)$ for $i=3$ and 4 by Lemma \ref{l74}. Moreover, $N_p(1)\geq N_p(-1)+1$ and $N_q(1) \geq N_q(-1)+1$. Hence, by Lemma \ref{l24} for 1, $N_r(-1) \geq N_r(1)+2$. Considering Lemma \ref{l24} for each integer, one can show that the weights are
\begin{center}
$\Sigma_{p}=\{-c,-b,x,x,x_{3},x_{4},-y,1,-1,1,-1,1\}\cup\{-1,1\}^{w}$
$\Sigma_{q}=\{c,a,-x,-x,x_{3}-c,x_{4}-c,c-y,1,-1,1,-1,1\}\cup\{-1,1\}^{w}$
$\Sigma_{r}=\{-a,b,-x_3,-x_4,y,c-x_3,c-x_4,y-c,-1,-1,-f,f\}\cup\{-1,1\}^{w}$
\end{center}
for some odd natural number $f$. As above, $f=1$, which is a contradiction since it means that $\min_{\alpha \in M^{S^1}} \min \{ N_{\alpha}(-1), N_{\alpha}(1) \} \geq w+1$.

\item $x_{i} \neq c-x_{j}$, for all $i$ and $j$.

By Lemma \ref{l72}, $-x_i \in \Sigma_r\setminus(\{-a,b\}\cup\{-1,1\}^w)$ for all $i$. Also, by Lemma \ref{l73}, $y-c \in \Sigma_r\setminus(\{-a,b\}\cup\{-1,1\}^w)$.

We show that $x_i \neq x_j$ for $i \neq j$. Suppose that $x_i=x_j$ for some $i \neq j$. Then by Lemma \ref{l66}, $2x_i=2x_j=c$, hence $x_i=c-x_j$, which contradicts the assumption. Therefore, $x_i \neq x_j$ for $i \neq j$. Moreover, by Lemma \ref{l63}, $x_i \neq c-y$ for all $i$ and $j$.

First, suppose that $y=1$. By Lemma \ref{l66}, none of $x_i$'s, $c-x_i$'s, and $c-y$ is 1. Then by Lemma \ref{l74}, $c-x_i \in \Sigma_r\setminus(\{-a,b\}\cup\{-1,1\}^w)$ for all $i$. Moreover, $N_p(1)=N_p(-1)=N_q(1)=N_q(-1)+1=w+3$. Hence, by Lemma \ref{l24} for 1, $N_r(-1)=N_r(1)+1$. Hence, the weights are
\begin{center}
$\Sigma_{p}=\{-c,-b,x_{1},x_{2},x_{3},x_{4},-1,1,-1,1,-1,1\}\cup\{-1,1\}^{w}$
$\Sigma_{q}=\{c,a,x_{1}-c,x_{2}-c,x_{3}-c,x_{4}-c,c-1,1,-1,1,-1,1\}\cup\{-1,1\}^{w}$
$\Sigma_{r}=\{-a,b\}\cup\{-x_i\}_{i=1}^4\cup\{c-x_i\}_{i=1}^4\cup\{1-c,-1\}\cup\{-1,1\}^{w}$
\end{center}
Then we have that $N_p(1) \geq w+3$, $N_q(1) \geq w+3$, and $N_r(1)=w$, which is a contradiction by Lemma \ref{l69}.

Second, suppose that $c-x_{i}=1$ for some $i$. Without loss of generality, assume that $c-x_{4}=1$. By Lemma \ref{l66}, none of $x_i$'s, $c-x_j$'s, $y$, and $c-y$ is 1 for all $i$ and $j \neq 4$. Thus $y \in \Sigma_r\setminus(\{-a,b\}\cup\{-1,1\}^w)$ by Lemma \ref{l75} and $c-x_i \in \Sigma_r\setminus(\{-a,b\}\cup\{-1,1\}^w)$ for $i \neq 4$ by Lemma \ref{l74}. Moreover, $N_p(1)=N_p(-1)+1=N_q(1)=N_q(-1)=w+3$. Therefore, by Lemma \ref{l24} for 1, $N_r(-1)=N_r(1)+1$. Hence, the weights are
\begin{center}
$\Sigma_{p}=\{-c,-b,x_{1},x_{2},x_{3},x_4,-y,1,-1,1,-1,1\}\cup\{-1,1\}^{w}$
$\Sigma_{q}=\{c,a,x_{1}-c,x_{2}-c,x_{3}-c,-1,c-y,1,-1,1,-1,1\}\cup\{-1,1\}^{w}$
$\Sigma_{r}=\{-a,b\}\cup\{-x_i\}_{i=1}^4\cup\{c-x_i\}_{i=1}^3\cup\{y,y-c,-1\}\cup\{-1,1\}^{w}$
\end{center}
Then we have that $N_p(1) \geq w+3$, $N_q(1) \geq w+3$, and $N_r(1)=w$, which is a contradiction by Lemma \ref{l69}.

Finally, suppose that $y \neq 1$ and $c-x_{i}\neq1$ for all $i$. By Lemma \ref{l75}, $y \in \Sigma_r\setminus(\{-a,b\}\cup\{-1,1\}^w)$. Also, $c-x_i \in \Sigma_r\setminus(\{-a,b\}\cup\{-1,1\}^w)$ for all $i$ by Lemma \ref{l74}. Hence, the weights are
\begin{center}
$\Sigma_{p}=\{-c,-b,x_{1},x_{2},x_{3},x_{4},-y,1,-1,1,-1,1\}\cup\{-1,1\}^{w}$
$\Sigma_{q}=\{c,a,x_{1}-c,x_{2}-c,x_{3}-c,x_{4}-c,c-y,1,-1,1,-1,1\}\cup\{-1,1\}^{w}$
$\Sigma_{r}=\{-a,b\}\cup\{-x_i\}_{i=1}^4\cup\{y\}\cup\{c-x_i\}_{i=1}^4\cup\{y-c\}\cup\{-1,1\}^{w}$
\end{center}
Then we have that $N_p(1) \geq w+3$, $N_q(1) \geq w+3$, and $N_r(1)=w$, which is a contradiction by Lemma \ref{l69}.
\end{enumerate}
\end{proof}

\begin{lem} \label{l87} In Lemma \ref{l71}, $t=2$ and $v=1$ are impossible.
\end{lem}

\begin{proof}
The weights in this case are
\begin{center}
$\Sigma_{p}=\{-c,-b,x_{1},x_{2},x_{3},x_{4},x_{5},-y_{1},-y_{2},1\}\cup\{-1,1\}^{w+1}$
$\Sigma_{q}=\{c,a,x_{1}-c,x_{2}-c,x_{3}-c,x_{4}-c,x_{5}-c,c-y_{1},c-y_{2},1\}\cup\{-1,1\}^{w+1}$
$\Sigma_{r}=\{-a,b,\cdots\}\cup\{-1,1\}^{w},$
\end{center}
where $a,b,$ and $c$ are even natural numbers such that $c=a+b$ is the largest weight, $x_i$'s and $y_i$'s are odd natural numbers for all $i$, and $w=\min_{\alpha \in M^{S^1}} \min \{ N_{\alpha}(-1), N_{\alpha}(1) \}$. Moreover, the remaining weights at $r$ are odd.

By Lemma \ref{l76}, $x_i=c-x_j$ for some $i$ and $j$. Then by Lemma \ref{l64}, there exist $x_i$ and $x_j$ where $i \neq j$ such that $2x_i=2x_j=c$. Without loss of generality, let $2x_1=2x_2=c$. Denote $x=x_1$. Lemma \ref{l64} also implies that $x_i \neq c-x_j$ for $i \neq 1$ and 2, and for all $j$. Therefore, $-x_i \in \Sigma_r\setminus(\{-a,b\}\cup\{-1,1\}^w)$ for $i \neq 1$ and 2 by Lemma \ref{l72}. Also, by Lemma \ref{l73}, $y_i-c \in \Sigma_r\setminus(\{-a,b\}\cup\{-1,1\}^w)$ for all $i$. Moreover, by Lemma \ref{l46} part 3 for $x$, none of $x_i$'s, $y_j$'s, $c-x_i$'s, and $c-y_j$'s can be $x$ for $i \neq 1$ and 2, and for all $j$. Hence, by Lemma \ref{l66}, all of $x_i$'s, $y_j$'s, $c-x_i$'s, and $c-y_j$'s are different for $i \neq 1$ and 2, and for all $j$.

First, suppose that $y_i=1$ for some $i$. Without loss of generality, let $y_2=1$. Then none of $x_i$'s, $c-x_i$'s, $y_1$, and $c-y_i$'s is 1 for all $i$. Then we have that $y_1 \in \Sigma_r\setminus(\{-a,b\}\cup\{-1,1\}^w)$ by Lemma \ref{l75} and $c-x_i \in \Sigma_r\setminus(\{-a,b\}\cup\{-1,1\}^w)$ for $i \neq 1$ and 2 by Lemma \ref{l74}. Moreover, $N_p(1)=N_p(-1)$ and $N_q(1)=N_q(-1)+1$. Hence, by Lemma \ref{l24} for 1, $N_r(1)+1=N_r(-1)$. Then the weights are
\begin{center}
$\Sigma_{p}=\{-c,-b,x,x,x_{3},x_{4},x_{5},-y_1,-1,1,-1,1\}\cup\{-1,1\}^{w}$
$\Sigma_{q}=\{c,a,-x,-x,x_{3}-c,x_{4}-c,x_{5}-c,c-y_1,c-1,1,-1,1\}\cup\{-1,1\}^{w}$
$\Sigma_{r}=\{-a,b,-x_3,-x_4,-x_5,y_1,c-x_3,c-x_4,c-x_5,y_1-c,1-c,-1\}\cup\{-1,1\}^{w}$
\end{center}
Then $\Sigma_{q} \equiv \Sigma_{r} \mod c-1$ by Lemma \ref{l46} part 1 for $c-1$. First, $|1+k(c-1)| > c$ for $|k| \geq 2$. Also, $c \notin \Sigma_r$. Then, $N_q(1)=w+2$, $N_r(1)=w$, and $\Sigma_{q} \equiv \Sigma_{r} \mod c-1$ imply that $N_r(2-c)=2$, which is a contradiction since $r$ has only one negative even weight.

Second, suppose that $c-x_i=1$ for some $i$. Since $c-x_1=c-x_2=x_1=\frac{c}{2}\geq2$, without loss of generality, let $c-x_5=1$. Then none of $x_i$'s, $c-x_j$'s, $y_i$'s, and $c-y_i$'s is 1 for all $i$ and $j \neq 5$. Therefore, we have that $y_i \in \Sigma_r\setminus(\{-a,b\}\cup\{-1,1\}^w)$ for all $i$ by Lemma \ref{l75} and $c-x_j \in \Sigma_r\setminus(\{-a,b\}\cup\{-1,1\}^w)$ for $j=3$ and 4 by Lemma \ref{l74}. Moreover, $N_p(1)=N_p(-1)+1$ and $N_q(1)=N_q(-1)$. Hence, by Lemma \ref{l24} for 1, $N_r(1)+1=N_r(-1)$. Then the weights are
\begin{center}
$\Sigma_{p}=\{-c,-b,x,x,x_{3},x_{4},c-1,-y_1,-y_2,1,-1,1\}\cup\{-1,1\}^{w}$
$\Sigma_{q}=\{c,a,-x,-x,x_{3}-c,x_{4}-c,-1,c-y_1,c-y_2,1,-1,1\}\cup\{-1,1\}^{w}$
$\Sigma_{r}=\{-a,b,-x_3,-x_4,1-c,y_1,y_2,c-x_3,c-x_4,y_1-c,y_2-c,-1\}\cup\{-1,1\}^{w}$
\end{center}
Then $\Sigma_{p} \equiv \Sigma_{r} \mod c-1$ by Lemma \ref{l46} part 1 for $c-1$. As above, $N_p(1)=w+2$, $N_r(1)=w$, and $\Sigma_{p} \equiv \Sigma_{r} \mod c-1$ imply that $N_r(2-c)=2$, which is a contradiction since $r$ has only one negative even weight.

Last, suppose that $c-x_i \neq 1$ and $y_i \neq 1$ for all $i$. Then we have that $y_i \in \Sigma_r\setminus(\{-a,b\}\cup\{-1,1\}^w)$ for all $i$ by Lemma \ref{l75} and $c-x_j \in \Sigma_r\setminus(\{-a,b\}\cup\{-1,1\}^w)$ for $j \neq 1$ and 2 by Lemma \ref{l74}. Then the weights are
\begin{center}
$\Sigma_{p}=\{-c,-b,x,x,x_{3},x_{4},x_{5},-y_1,-y_2,1,-1,1\}\cup\{-1,1\}^{w}$
$\Sigma_{q}=\{c,a,-x,-x,x_{3}-c,x_{4}-c,x_{5}-c,c-y_1,c-y_2,1,-1,1\}\cup\{-1,1\}^{w}$
$\Sigma_{r}=\{-a,b,-x_3,-x_4,-x_5,y_1,y_2,c-x_3,c-x_4,c-x_5,y_1-c,y_2-c\}\cup\{-1,1\}^{w}$
\end{center}
Then we have that $\lambda_r=\frac{1}{2}\dim M$, which contradicts Lemma \ref{l33} that $\lambda_r=\frac{1}{2}\dim M+2$.
\end{proof}

\begin{lem} \label{l88} In Lemma \ref{l71}, $t=2$ and $v=2$ are impossible.
\end{lem}

\begin{proof}
In this case, the weights are
\begin{center}
$\Sigma_{p}=\{-c,-b,x_1,x_2,x_{3},x_4,x_5,-y_1,-y_2,1\}\cup\{-1,1\}^{w+2}$
$\Sigma_{q}=\{c,a,x_1-c,x_2-c,x_{3}-c,x_4-c,x_5-c,c-y_1,c-y_2,1\}\cup\{-1,1\}^{w+2}$
$\Sigma_{r}=\{-a,b,\cdots\}\cup\{-1,1\}^w,$
\end{center}
where $a,b,$ and $c$ are even natural numbers such that $c=a+b$ is the largest weight, $x_i$'s and $y_i$'s are odd natural numbers for all $i$, and $w=\min_{\alpha \in M^{S^1}} \min \{ N_{\alpha}(-1), N_{\alpha}(1) \}$. Moreover, the remaining weights at $r$ are odd.

By Lemma \ref{l76}, $x_i=c-x_j$ for some $i$ and $j$. Then by Lemma \ref{l64}, there exist $x_i$ and $x_j$ where $i \neq j$ such that $2x_i=2x_j=c$. Without loss of generality, let $2x_1=2x_2=c$. Denote $x=x_1$. Lemma \ref{l64} also implies that $x_i \neq c-x_j$ for $i \neq 1$ and 2, and for all $j$. Therefore, $-x_i \in \Sigma_r\setminus(\{-a,b\}\cup\{-1,1\}^w)$ for $i \neq 1$ and 2 by Lemma \ref{l72}. Also, by Lemma \ref{l73}, $y_i-c \in \Sigma_r\setminus(\{-a,b\}\cup\{-1,1\}^w)$ for all $i$. Moreover, by Lemma \ref{l46} part 3 for $x$, none of $x_i$'s, $y_j$'s, $c-x_i$'s, and $c-y_j$'s can be $x$ for $i \neq 1$ and 2, and for all $j$. Hence, by Lemma \ref{l66}, all of $x_i$'s, $y_j$'s, $c-x_i$'s, and $c-y_j$'s are different for $i \neq 1$ and 2, and for all $j$.

First, suppose that $y_i=1$ for some $i$. Without loss of generality, let $y_2=1$. Then none of $x_i$'s, $c-x_i$'s, $y_1$, and $c-y_i$'s is 1 for all $i$. Therefore, we have that $y_1 \in \Sigma_r\setminus(\{-a,b\}\cup\{-1,1\}^w)$ by Lemma \ref{l75} and $c-x_i \in \Sigma_r\setminus(\{-a,b\}\cup\{-1,1\}^w)$ for $i \neq 1$ and 2 by Lemma \ref{l74}. Moreover, $N_p(1)=N_p(-1)$ and $N_q(1)=N_q(-1)+1$. Hence, by Lemma \ref{l24} for 1, $N_r(1)+1=N_r(-1)$. Considering Lemma \ref{l24} for each integer, the weights are 
$$\Sigma_{p}=\{-c,-b,x,x,x_{3},x_{4},x_{5},-y_1,-1,1,-1,1,-1,1\}\cup\{-1,1\}^{w}$$
$$\Sigma_{q}=\{c,a,-x,-x,x_{3}-c,x_{4}-c,x_{5}-c,c-y_1,c-1,1,-1,1,-1,1\}\cup\{-1,1\}^{w}$$
$$\Sigma_{r}=\{-a,b\}\cup\{-x_i\}_{i=3}^5\cup\{y_1\}\cup\{c-x_i\}_{i=3}^5\cup\{y_i-c\}_{i=1}^2\cup\{-1,-f,f\}\cup\{-1,1\}^{w}$$
for some odd natural number $f$. If $f>1$, by Lemma \ref{l68}, $c=2x=2f$, which is a contradiction that no additional multiples of $x$ should appear as weights by Lemma \ref{l46} part 3 for $x$. Hence $f=1$, which is a contradiction since it means that $\min_{\alpha \in M^{S^1}} \min \{ N_{\alpha}(-1), N_{\alpha}(1) \} \geq w+1$.

Second, suppose that $c-x_i=1$ for some $i$. Since $c-x_1=c-x_2=x_1=\frac{c}{2} \geq2$, without loss of generality, let $c-x_5=1$.
Then none of $x_i$'s, $c-x_j$'s, $y_i$'s, and $c-y_i$'s is 1 for all $i$ and $j \neq 5$. Therefore, we have that $y_i \in \Sigma_r\setminus(\{-a,b\}\cup\{-1,1\}^w)$ for all $i$ by Lemma \ref{l75} and $c-x_j \in \Sigma_r\setminus(\{-a,b\}\cup\{-1,1\}^w)$ for $j=3$ and 4 by Lemma \ref{l74}. Moreover, $N_p(1)=N_p(-1)$ and $N_q(1)=N_q(-1)+1$. Hence, by Lemma \ref{l24} for 1, $N_r(1)+1=N_r(-1)$. Then the weights are
$$\Sigma_{p}=\{-c,-b,x,x,x_{3},x_{4},c-1,-y_1,-y_2,1,-1,1,-1,1\}\cup\{-1,1\}^{w}$$
$$\Sigma_{q}=\{c,a,-x,-x,x_{3}-c,x_{4}-c,-1,c-y_1,c-y_2,1,-1,1,-1,1\}\cup\{-1,1\}^{w}$$
$$\Sigma_{r}=\{-a,b\}\cup\{-x_i\}_{i=3}^5\cup\{y_i\}_{i=1}^2\cup\{c-x_i\}_{i=3}^4\cup\{y_i-c\}_{i=1}^2\cup\{-1,-f,f\}\cup\{-1,1\}^{w}$$
for some odd natural number $f$. As above, $f=1$, which is a contradiction since it means that $\min_{\alpha \in M^{S^1}} \min \{ N_{\alpha}(-1), N_{\alpha}(1) \} \geq w+1$.

Last, suppose that $c-x_i \neq 1$ and $y_i \neq 1$ for all $i$. Then we have that $y_i \in \Sigma_r\setminus(\{-a,b\}\cup\{-1,1\}^w)$ for all $i$ by Lemma \ref{l75} and $c-x_j \in \Sigma_r\setminus(\{-a,b\}\cup\{-1,1\}^w)$ for $j \neq 1$ and 2 by Lemma \ref{l74}. Since $N_p(1)\geq N_p(-1)+1$ and $N_q(1) \geq N_q(-1)+1$, by Lemma \ref{l24} for 1, $N_r(1)+2 \leq N_r(-1)$. Then the weights are
$$\Sigma_{p}=\{-c,-b,x,x,x_{3},x_{4},x_{5},-y_1,-y_2,1,-1,1,-1,1\}\cup\{-1,1\}^{w}$$
$$\Sigma_{q}=\{c,a,-x,-x,x_{3}-c,x_{4}-c,x_{5}-c,c-y_1,c-y_2,1,-1,1,-1,1\}\cup\{-1,1\}^{w}$$
$$\Sigma_{r}=\{-a,b\}\cup\{-x_i\}_{i=3}^5\cup\{y_i\}_{i=1}^2\cup\{c-x_i\}_{i=3}^5\cup\{y_i-c\}_{i=1}^2\cup\{-1,-1\}\cup\{-1,1\}^{w}$$
Then we have that $N_p(1)\geq w+3$, $N_q(1)\geq w+3$, and $N_r(1)=w$, which is a contradiction by Lemma \ref{l69}.
\end{proof}

\begin{lem} \label{l89} In Lemma \ref{l71}, $t=3$ and $v=2$ are impossible.
\end{lem}

\begin{proof}
In this case, the weights 
\begin{center}
$\Sigma_{p}=\{-c,-b\}\cup\{x_i\}_{i=1}^6\cup\{-y_i\}_{i=1}^3\cup\{1\}\cup\{-1,1\}^{w+2}$
$\Sigma_{q}=\{c,a\}\cup\{x_i-c\}_{i=1}^6\cup\{c-y_i\}_{i=1}^3\cup\{1\}\cup\{-1,1\}^{w+2}$
$\Sigma_{r}=\{-a,b,\cdots\}\cup\{-1,1\}^w,$
\end{center}
where $a,b,$ and $c$ are even natural numbers such that $c=a+b$ is the largest weight, $x_i$'s and $y_i$'s are odd natural numbers for all $i$, and $w=\min_{\alpha \in M^{S^1}} \min \{ N_{\alpha}(-1), N_{\alpha}(1) \}$. Moreover, the remaining weights at $r$ are odd.

By Lemma \ref{l76}, $x_i=c-x_j$ for some $i$ and $j$. Then by Lemma \ref{l64}, there exist $x_i$ and $x_j$ where $i \neq j$ such that $2x_i=2x_j=c$. Without loss of generality, let $2x_1=2x_2=c$. Denote $x=x_1$. Lemma \ref{l64} also implies that $x_i \neq c-x_j$ for $i \neq 1$ and 2, and for all $j$. Therefore, $-x_i \in \Sigma_r\setminus(\{-a,b\}\cup\{-1,1\}^w)$ for $i \neq 1$ and 2 by Lemma \ref{l72}. Also, by Lemma \ref{l73}, $y_i-c \in \Sigma_r\setminus(\{-a,b\}\cup\{-1,1\}^w)$ for all $i$. Moreover, by Lemma \ref{l46} part 3 for $x$, none of $x_i$'s, $y_j$'s, $c-x_i$'s, and $c-y_j$'s can be $x$ for $i \neq 1$ and 2, and for all $j$. Hence, by Lemma \ref{l66}, all of $x_i$'s, $y_j$'s, $c-x_i$'s, and $c-y_j$'s are different for $i \neq 1$ and 2, and for all $j$. 

First, suppose that $y_i=1$ for some $i$. Without loss of generality, let $y_3=1$. Then none of $x_i$'s, $c-x_i$'s, $y_j$'s, and $c-y_i$'s is 1 for all $i$ and $j \neq 3$. Therefore, $y_i \in \Sigma_r\setminus(\{-a,b\}\cup\{-1,1\}^w)$ for $i\neq3$ by Lemma \ref{l75} and $c-x_j \in \Sigma_r\setminus(\{-a,b\}\cup\{-1,1\}^w)$ for $j \neq 1$ and 2 by Lemma \ref{l74}. Moreover, $N_p(1)=N_p(-1)$ and $N_q(1)=N_q(-1)+1$. Hence, by Lemma \ref{l24} for 1, $N_r(1)+1=N_r(-1)$. Then the weights are 
\begin{center}
$\Sigma_{p}=\{-c,-b,x,x,x_{3},x_{4},x_{5},x_6,-y_1,-y_2,-1,1\}\cup\{-1,1\}^{w+2}$
$\Sigma_{q}=\{c,a,-x,-x,x_{3}-c,x_{4}-c,x_{5}-c,x_6-c,c-y_1,c-y_2,c-1,1\}\cup\{-1,1\}^{w+2}$
$\Sigma_{r}=\{-a,b\}\cup\{-x_i\}_{i=3}^6\cup\{y_i\}_{i=1}^2\cup\{c-x_i\}_{i=3}^6\cup\{y_i-c\}_{i=1}^3\cup\{-1\}\cup\{-1,1\}^{w}$
\end{center}
Then we have that $N_p(1)=w+3$, $N_q(1)=w+3$, and $N_r(1)=w$, which is a contradiction by Lemma \ref{l69}.

Second, suppose that $c-x_i=1$ for some $i$. Since $c-x_1=c-x_2=x_1=\frac{c}{2}\geq2$, without loss of generality, let $c-x_6=1$. Then 
none of $x_i$'s, $c-x_j$'s, $y_i$'s, and $c-y_i$'s is 1 for all $i$ and $j \neq 6$. Therefore, $y_i \in \Sigma_r\setminus(\{-a,b\}\cup\{-1,1\}^w)$ for all $i$ by Lemma \ref{l75} and $c-x_j \in \Sigma_r\setminus(\{-a,b\}\cup\{-1,1\}^w)$ for $j=3,4$, and 5 by Lemma \ref{l74}. Moreover, $N_p(1)=N_p(-1)+1$ and $N_q(1)=N_q(-1)$. Hence, by Lemma \ref{l24} for 1, $N_r(1)+1=N_r(-1)$. Then the weights are
\begin{center}
$\Sigma_{p}=\{-c,-b,x,x,x_{3},x_{4},x_{5},c-1,-y_1,-y_2,-y_3,1\}\cup\{-1,1\}^{w+2}$
$\Sigma_{q}=\{c,a,-x,-x,x_{3}-c,x_{4}-c,x_{5}-c,-1,c-y_1,c-y_2,c-y_3,1\}\cup\{-1,1\}^{w+2}$
$\Sigma_{r}=\{-a,b\}\cup\{-x_i\}_{i=3}^6\cup\{y_i\}_{i=1}^3\cup\{c-x_i\}_{i=3}^5\cup\{y_i-c\}_{i=1}^3\cup\{-1\}\cup\{-1,1\}^{w}$
\end{center}
Then we have that $N_p(1)=w+3$, $N_q(1)=w+3$, and $N_r(1)=w$, which is a contradiction by Lemma \ref{l69}.

Last, suppose that $c-x_i \neq 1$ and $y_i \neq 1$ for all $i$. Then we have that $y_i \in \Sigma_r\setminus(\{-a,b\}\cup\{-1,1\}^w)$ for all $i$ by Lemma \ref{l75} and $c-x_j \in \Sigma_r\setminus(\{-a,b\}\cup\{-1,1\}^w)$ for $j \neq 1$ and 2 by Lemma \ref{l74}. Then the weights are
$$\Sigma_{p}=\{-c,-b,x,x,x_{3},x_{4},x_{5},x_6,-y_1,-y_2,-y_3,1\}\cup\{-1,1\}^{w+2}$$
$$\Sigma_{q}=\{c,a,-x,-x,x_{3}-c,x_{4}-c,x_{5}-c,x_6-c,c-y_1,c-y_2,c-y_3,1\}\cup\{-1,1\}^{w+2}$$
$$\Sigma_{r}=\{-a,b\}\cup\{-x_i\}_{i=3}^6\cup\{y_i\}_{i=1}^3\cup\{c-x_i\}_{i=3}^6\cup\{y_i-c\}_{i=1}^3\cup\{-1,1\}^{w}$$
Then we have that $N_p(1)\geq w+3$, $N_q(1)\geq w+3$, and $N_r(1)=w$, which is a contradiction by Lemma \ref{l69}.
\end{proof}

\begin{lem} \label{l810} In Lemma \ref{l71}, $t \geq v+2$ is impossible.
\end{lem}

\begin{proof}
By Lemma \ref{l76}, $x_i=c-x_j$ for some $i$ and $j$. Then by Lemma \ref{l64}, there exist $x_i$ and $x_j$ where $i \neq j$ such that $2x_i=2x_j=c$. Without loss of generality, let $2x_1=2x_2=c$. Lemma \ref{l64} also implies that $x_i \neq c-x_j$ for $i \neq 1$ and 2, and for all $j$. Therefore, $-x_i \in \Sigma_r\setminus(\{-a,b\}\cup\{-1,1\}^w)$ for $i \neq 1$ and 2 by Lemma \ref{l72}. Also, by Lemma \ref{l73}, $y_i-c \in \Sigma_r\setminus(\{-a,b\}\cup\{-1,1\}^w)$ for all $i$. Moreover, by Lemma \ref{l46} part 3 for $x$, none of $x_i$'s, $y_j$'s, $c-x_i$'s, and $c-y_j$'s can be $x$ for $i \neq 1$ and 2, and for all $j$. Hence, by Lemma \ref{l66}, all of $x_i$'s, $y_j$'s, $c-x_i$'s, and $c-y_j$'s are different for $i \neq 1$ and 2, and for all $j$. 

First, suppose that $c-x_i=1$ for some $i$. Since $c-x_1=c-x_2=x_1=\frac{c}{2}\geq2$, without loss of generality, let $c-x_3=1$. Then, by Lemma \ref{l66}, $c-x_i \neq 1$ for $i \neq 3$ and $y_j \neq 1$ for all $j$. Then $c-x_i \in \Sigma_r(\{-a,b\} \cup \{-1,1\}^w)$ for $i\neq 1,2$, and 3 by Lemma \ref{l74}, and $y_j \in \Sigma_(\{-a,b\} \cup \{-1,1\}^w)$ for all $j$ by Lemma \ref{l73}. Therefore, we have that $\{-x_i\}_{i=3}^{t+3} \cup \{y_i\}_{i=1}^{t} \cup \{c-x_i\}_{i=4}^{t+3} \cup \{y_i-c\}_{i=1}^{t} \subset \Sigma_r \setminus (\{-a,b\} \cup \{-1,1\}^w)$, which is a contradiction since there are $2t+4+2u+2v$ spaces in $\Sigma_r \setminus (\{-a,b\} \cup \{-1,1\}^w)$ but $4t+1 > 4t \geq 2t+4+2u+2v$ by the assumption.

Second, suppose that $y_i=1$ for some $i$. Without loss of generality, let $y_1=1$. Then, by Lemma \ref{l66}, $c-x_i \neq 1$ for all $i$ and $y_j \neq 1$ for $j\neq1$. Then $c-x_i \in \Sigma_r\setminus(\{-a,b\} \cup \{-1,1\}^w)$ for $i\neq 1$ and 2 by Lemma \ref{l74}, and $y_j \in \Sigma_r\setminus(\{-a,b\} \cup \{-1,1\}^w)$ for $j\neq1$ by Lemma \ref{l73}. Therefore, we have that $\{-x_i\}_{i=3}^{t+3} \cup \{y_i\}_{i=2}^{t} \cup \{c-x_i\}_{i=3}^{t+3} \cup \{y_i-c\}_{i=1}^{t} \subset \Sigma_r \setminus (\{-a,b\} \cup \{-1,1\}^w)$, which is a contradiction since there are $2t+4+2u+2v$ spaces in $\Sigma_r \setminus (\{-a,b\} \cup \{-1,1\}^w)$ but $4t+1 > 4t \geq 2t+4+2u+2v$ by the assumption.

Last, suppose that $c-x_i \neq 1$ and $y_i \neq 1$ for all $i$. Then $c-x_i \in \Sigma_r\setminus(\{-a,b\} \cup \{-1,1\}^w)$ for $i\neq 1$ and 2 by Lemma \ref{l74}, and $y_j \in \Sigma_r\setminus(\{-a,b\} \cup \{-1,1\}^w)$ for all $j$ by Lemma \ref{l73}. Therefore, we have that $\{-x_i\}_{i=3}^{t+3} \cup \{y_i\}_{i=1}^{t} \cup \{c-x_i\}_{i=3}^{t+3} \cup \{y_i-c\}_{i=1}^{t} \subset \Sigma_r \setminus (\{-a,b\} \cup \{-1,1\}^w)$, which is a contradiction since there are $2t+4+2u+2v$ spaces in $\Sigma_r \setminus (\{-a,b\} \cup \{-1,1\}^w)$ but $4t+2 > 4t \geq 2t+4+2u+2v$ by the assumption.
\end{proof}

\begin{lem} \label{l811} In Lemma \ref{l71}, $v\geq 3$ are impossible.
\end{lem}

\begin{proof}
First, $\min\{N_p(-1),N_p(1)\} \geq w+3$ and $\min\{N_q(-1),N_q(1)\} \geq w+3$. If $\min\{N_r(-1),N_r(1)\} > w$, then $\min_{\alpha \in M^{S^1}} \min \{ N_{\alpha}(-1), N_{\alpha}(1) \} \geq w+1$, which is a contradiction. Hence $\min\{N_r(-1),N_r(1)\} = w$. Then either $N_r(-1)=w$ or $N_r(1)=w$. However, neither case is possible by Lemma \ref{l69}.
\end{proof}

\section*{Acknowledgements}
I would like to thank Sue Tolman, my advisor, for her kind advice and help.

\end{document}